\newcommand{\zed}{\Xi}

\newcommand{\INT}{[0,1]}

\newcommand{\bz}{\bf z}
\newcommand{\bg}{\bf g}
\newcommand{\DB}{\mathcal{D}}
\newcommand{\CK}{\mathcal{K}}
\newcommand{\mL}{\mathcal{L}}
\newcommand{\mZ}{\mathbb{Z}}

\newcommand\supp{\operatorname{supp}}

\documentclass[11pt,fleqn,english,reqno]{amsart}
\newcommand\del[1]{}

\usepackage{amsmath}   
\usepackage{amsfonts,amssymb}
\usepackage[usenames]{color}
\usepackage[latin1]{inputenc}
\usepackage{enumitem}
\usepackage{amssymb}
\usepackage{babel}
\usepackage{amsfonts}
\usepackage{latexsym}
\usepackage{mathrsfs}
{
\definecolor{sonst}{cmyk}{0.4,1.0,0.1,0.3}
\definecolor{yellow}{cmyk}{0.1,0.2,0.9,0.3}
\definecolor{darkblue}{rgb}{0.0,0.1,0.5}
}

\input colordvi

\newcommand{\leb}{\mbox{Leb}}
\newcommand{\mT}{\mathbb{T}}

\newcommand{\levy}{L\'evy }

\newcommand{\barray}{\begin{array}{rcl}}
\newcommand{\earray}{\end{array}}

\newcommand{\spn}{\mbox{span}}
\newcommand{\bcase}{\begin{cases}}
\newcommand{\ecase}{\end{cases}}
\newcommand{\barr}{\lk\{\begin{array}{rcl}}
\newcommand{\earr}{\end{array}\rk.}

\newcommand{\CN}{{{ \mathcal N }}}

\newcommand{\bu}{{{ \bf u }}}
\newcommand{\bv}{{{ \bf v }}}
\newcommand{\bx}{{{ \bf x }}}
\newcommand{\by}{{{ \bf y }}}

\newcommand{\CP}{{{ \mathcal P }}}

\newcommand{\CB}{{{ \mathcal B }}}
\newcommand{\CS}{{{ \mathcal S }}}
\newcommand{\CI}{{{ \mathcal I }}}

\newcommand{\CT}{{{ \mathcal T }}}

\newcommand{\QQ}{\mathbb{Q}}

\newcommand{\pmat}{\begin{pmatrix}}
\newcommand{\epmat}{\end{pmatrix}}
\newcommand{\lqq}{\lefteqn }

\newcommand{\CO}{{\mathcal{O}}}
\newcommand{\CR}{{\mathcal{R}}}

\newcommand{\lk}{\left}
\newcommand{\rk}{\right}

\newcommand{\ep} {\epsilon }

\newcommand{\be} {\begin{enumerate} }
\newcommand{\ee} {\end{enumerate} }

\newcommand{\CF}{{ \mathcal{F } }}

\newcommand{\h}{\mathscr{H}}
\newcommand{\rA}{\mathcal{L}}
\newcommand{\ve}{\mathbf{V}}
\newcommand{\mo}{\mT}
\newcommand{\el}{L}
\newcommand{\bbop}{B}
\newcommand{\bw}{\mathbf{w}}
\newcommand{\RR}{{\mathbb{R}}}
\newcommand{\DD}{{\rm I \kern -0.2em D}}
\newcommand{\dd}{{\rm I \kern -0.2em D}}
\newcommand{\NN}{{\rm I \kern -0.2em N}}

\newcommand{\PP}{{\mathbb{P}}}

\newcommand{\EE}{ \mathbb{E} }
\newcommand{\TT}{{\rm I \kern -0.2em T}}

\newcommand{\DEQS}{\begin{eqnarray*}}
\newcommand{\EEQS}{\end{eqnarray*}}
\newcommand{\DEQSZ}{\begin{eqnarray}}
\newcommand{\EEQSZ}{\end{eqnarray}}

\newcommand{\beq}{\begin{equation}}
\newcommand{\eeq}{\end{equation}}

\begin{document}

\def\JCMvol{xx}
\def\JCMno{x}
\def\JCMyear{200x}
\def\JCMreceived{Month xx, 200x}
\def\JCMrevised{}
\def\JCMaccepted{}
\setcounter{page}{1}



\def\cgeq{\raisebox{-1mm}{$\;\stackrel{>}{\sim}\;$}}
\def\cequiv{\raisebox{-1mm}{$\;\stackrel{=}{\sim}\;$}}
\def\cleq{\raisebox{-1mm}{$\;\stackrel{<}{\sim}\;$}}
\def\T{\mathop{\cal T}}
\def\esssup{\mathop{\rm esssup}}
\def\qed{\hfill$\fbox{}$}
\catcode`@=11
\newskip\plaincentering \plaincentering=0pt plus 1000pt minus 1000pt
\def\@plainlign{\tabskip=0pt\everycr={}} 
\def\eqalignno#1{\displ@y \tabskip=\plaincentering
  \halign to\displaywidth{\hfil$\@plainlign\displaystyle{##}$\tabskip=0pt
    &$\@plainlign\displaystyle{{}}##$\hfil\tabskip=\plaincentering
    &\llap{$\hbox{\rm\@plainlign##}$}\tabskip=0pt\crcr
    #1\crcr}}



%
%
%
%
%
\newcounter{gr1nn}
\newenvironment{steplist}
{\begin{list} {  ${\bf \circ}$ Step \arabic{gr1nn}:}
{\usecounter{gr1nn}
\setlength{\leftmargin}{0.9cm}
\setlength{\topsep}{0.1cm} \setlength{\itemsep}{0.2cm}
\setlength{\parsep}{0.1cm} \setlength{\itemindent}{-0.3cm}
\setlength{\parskip}{0.1cm}}} {\end{list}}
\newcounter{gr1n}
\newenvironment{numlistn}
{\begin{list} { (\roman{gr1n})}
{\usecounter{gr1n}
\setlength{\topsep}{0.1cm}
\setlength{\itemsep}{0.0cm}
\setlength{\parsep}{0.1cm}
\setlength{\itemindent}{-0.7cm}
\setlength{\parskip}{0.0cm}}}
{\end{list}}
%




\def\cgeq{\raisebox{-1mm}{$\;\stackrel{>}{\sim}\;$}}
\def\cequiv{\raisebox{-1mm}{$\;\stackrel{=}{\sim}\;$}}
\def\cleq{\raisebox{-1mm}{$\;\stackrel{<}{\sim}\;$}}
\def\T{\mathop{\cal T}}
\def\esssup{\mathop{\rm esssup}}
\def\qed{\hfill$\fbox{}$}
\catcode`@=11
\newskip\plaincentering \plaincentering=0pt plus 1000pt minus 1000pt
\def\@plainlign{\tabskip=0pt\everycr={}} 
\def\eqalignno#1{\displ@y \tabskip=\plaincentering
  \halign to\displaywidth{\hfil$\@plainlign\displaystyle{##}$\tabskip=0pt
    &$\@plainlign\displaystyle{{}}##$\hfil\tabskip=\plaincentering
    &\llap{$\hbox{\rm\@plainlign##}$}\tabskip=0pt\crcr
    #1\crcr}}




\author{E. Hausenblas}

\address{Department of Mathematics\\ Montanuniversity Leoben\\
Fr. Josefstr. 18\\ 8700 Leoben, Austria}
 \email{erika.hausenblas@unileoben.ac.at}
\author{ Paul A. Razafimandimby}
\address{Department of Mathematics and Applied Mathematics\\ University of Pretoria\\ Lynwood Road \\ Hatfield, Pretoria 0083, South Africa}
 \email{paul.razafimandimby@up.ac.za}

\date{\today}

 \title[Irreducibility and the 2Dim Stochastic Navier Stokes]{Existence of a density 
 of the 2Dim Stochastic Navier Stokes Equation driven by L\'evy processes or fractional Brownian motion}

\newtheorem{notation}{Notation}[section]

\newtheorem{assumption}{Assumption}[section]
\newtheorem{claim}{Claim}[section]
\newtheorem{lemma}{Lemma}[section]
\newtheorem{example}{Example}[section]
\newtheorem{tlemma}{Technical Lemma}[section]
\newtheorem{definition}{Definition}[section]
\newtheorem{remark}{Remark}[section]
\newtheorem{cor}{Corollary}[section]
\newtheorem{proposition}{Proposition}[section]
\newtheorem{theorem}{Theorem}[section]
\newtheorem{algorithm}{Algorithm}
\newtheorem{question}{Question}






\begin{abstract}
In this article we are interested in the regularity properties of the
probability measure induced by the solution process of the L\'evy noise or a fractional Brownian motion driven Navier Stokes Equation on the two dimensional torus $\mathbb{T}$.  We mainly investigate under which conditions on the characteristic measure of the L\'evy process or the Hurst parameter of the  fractal Brownian motion the law of the projection of  $u(t)$ onto any finite dimensional $F\subset L^2(\mathbb{T})$   is absolutely continuous with respect to the Lebesgue measure on $F$.
%
\end{abstract}

\del{

}

\maketitle

\section{Introduction}

We consider the Navier-Stokes equations (NSEs) subjected to the periodic boundary condition on the torus 
\begin{equation}  \label{eqn:4.1}
\hspace{2cm}\begin{cases}
\partial_tu(t) -\nu \Delta u(t) + u(t) \cdot \nabla u(t) + \nabla \mathfrak{p}(t) = \dot{\Xi}(t),\\
\nabla\cdot u(t)=0,\\
u(0)=u_0,
\end{cases}%
\end{equation}
where $u$ and $\mathfrak{p}$ are unknown vector field and scalar periodic functions in the space variable representing
 the fluid velocity and the pressure, respectively.  We assume that we are given an initial velocity $u_0$. The perturbation $\dot{\Xi}$ denotes, roughly speaking,  the Radon-Nikodym derivative of a \levy process $\Xi=L$ or a fractional Brownian motion $\Xi=B^H$. In the case when $\Xi$ is a Wiener noise the above system has been the subject of intensive mathematical studies since the pioneering work of Bensoussan and Temam. The analysis of the qualitative properties and long time behaviour of its solutions has generated several important results, see for instance \cite{brz1,AD-chapter,franco,martin1,SK+AS}, to cite a few results.
  Particularly, when
\DEQSZ 
\label{eqn:decompose}
 \Xi&=&\sum_{j=1}^\infty b_j \beta_j e_j,
\EEQSZ 
 where $(b_j)_{j \in \mathbb{N}}$ is a sequence of non-negative numbers, $(\beta_j)_{j \in \mathbb{N}}$ is a sequence of independent, identically distributed real-valued Brownian motions and $(e_j)_{j \in \mathbb{N}}$ is an orthonormal basis of the space of square integrable, periodic and divergence free functions with mean zero, the authors in \cite{deb}, \cite{armen} and \cite{mat} proved the existence of densities for the laws of finite dimensional functionals of its solutions. In these papers different methods are used to prove the existence of such densities, for instance in \cite{deb} a method based on Girsanov theorem is used and the Malliavin calculus is used in \cite{mat}. In \cite{armen} a method based on controllability of \eqref{eqn:4.1} in finite-dimensional projections and an abstract result on image of decomposable measure under analytic mappings is used.
 This method does not use the Gaussian structure of the noise as the methods in  \cite{deb} and \cite{mat}. In this paper we are mainly interested in proving the existence of densities for the laws of finite-dimensional analytic functionals of the solution of  \eqref{eqn:4.1} when the driving noise $\Xi$ is a L\'evy noise or a fractional Brownian motion. For this purpose we extend the results in  \cite{armen} to our framework.  Although we closely follow the approach in \cite{armen} the extension of the result therein to our setting is not trivial. In fact, the proof in \cite{armen} relies very much on the natural decomposability of the driving noise law in a Hilbert space $\mathscr{H}$ which is not naturally satisfied by a L\'evy process or a fractional Brownian. In fact, even if the L\'evy  noise (or fractional Brownian motion) $\Xi$ has a decomposition as in \eqref{eqn:decompose}, which is one of the main assumptions in \cite{armen}, it is not known whether  there exists a Hilbert space $\mathscr{H}$ on which the law of $\Xi$ on $\mathscr{H}$ is decomposable. In order to overcome this difficulty we prove,  by using wavelet analysis and the decomposability of measure on Banach space introduced in \cite{dineen}, that there exists a Banach space $\mathscr{H}$ with Schauder basis on which the law of $\Xi$ is decomposable. With this result at hand and using the solid controllability of  \eqref{eqn:4.1} we can prove the existence of densities for the laws of  finite-dimensional projection of the solutions of \eqref{eqn:4.1}.

{In the next section we will fix the notation and present some preliminary results. Section 3 is devoted to the statement and the proof of our  main result which will be applied to the stochastic 2D Navier-Stokes equations in the torus. In Appendix A and Appendix B we present and prove several results related to the wavelet expansion of L\'evy noise and fractional
Brownian motion, respectively.  In Appendix C  we establish a zero one law result, which is crucial for the proof of the main result, for decomposable measures.}

\section{Notations, Hypotheses and preliminary results}

For a separable Banach space $E$ we denote by
$\CB(E)$ its Borel $\sigma$--algebra. For a subspace $E_0$  of $E$ we denote by $E_1$ the subspace of $E$ such that   $E=E_0 \oplus E_1$, \textit{i.e.}, $E_1=E_0^{\perp}$. Furthermore,
%
for $A\subset E$ and $y\in E_1$ we set $$A_{(E_0,E_1)}(y)=\{ x\in E_0: x+y\in A\}.$$
Let $\mu$ be a probability measure on $(E,\CB(E))$ and $E_0$ and $E_1$ as above. We define a probability measure $\mu_{E_0}$ on $(E_0, \mathcal{B}(E_0))$ by
$$
\mu_{E_0}:\CB(E_0)\ni A\mapsto \mu(A+E_1)\in[0,1].
$$
{For a subspace $\tilde{E}_0\subset E_1$ we set}
 $$
\mu_{(\tilde{E}_0,E_1)}:\CB(\tilde{E}_0)\ni A\mapsto \mu(A+E_1)\in[0,1].
$$
If $E_0$ is finite dimensional, then we denote by  $\leb_{E_0}$ the measure defined by
$$
\leb_{E_0}: \CB({E_0}) \ni U \mapsto \mu_{E_0}(U):= \leb_{\RR^n}(\iota^{-1}(U)), $$
where $\iota$  is the isomorphism $\iota:E_0\to \RR^n$, $n=\dim(E_0)$.

\medskip
We can now introduce the following definition.
\begin{definition}\label{decomposkerneldef}
	Let $\{F_n:n\in\NN\}$ be a family of mutually disjoint closed subspaces of $E$, \text{i.e.} $F_j\cap F_k=\{0\}$, $j\not = k$. We set  $G_n := F_1\oplus\cdots\oplus F_n$ and $G^n := (F_1\oplus\cdots\oplus F_n)^ \perp$. If for any $n \in \mathbb{N}$ there exists a kernel
	$$
	l_n: G^n\times \CB( F_1\oplus \cdots \oplus F_n)\to \RR^+_0,
	$$
	such that
	$$
	\mu(A) = \int_{G^n}
	\int_{A_n(y)} l_n(\by,dx) \mu_{G^ n}(d\by),
	$$
	where $A_n(\by)=A_{(F_1\oplus\cdots\oplus F_n,G^n)}(\by)$, then we say that the measure $\mu$ is decomposable with  decomposition
	$\{F_n,G^n,l_n\}_{n=1}^\infty$.
\end{definition}
Hereafter we fix a
 separable Banach space $E$ with Schauder basis $\{ e_n:n\in\NN\}$ and we set
\DEQSZ 
\label{def:F_n}
 F_n=\{ \lambda e_n:\lambda\in \RR\}.
\EEQSZ
  We also set
\DEQSZ 
\label{def:G_n}
   G_n= F_0\oplus F_1\oplus\ldots \oplus F_n \text{ and }  G^ n=G_n^ \perp,
\EEQSZ 
along which we consider a probability kernel
$$ l_n:G^n\times \CB( G_n 
)\to [0,1]. $$
The projection onto any nontrivial subspace $F\subset E$ is denoted by $\pi_F$.
Having fixed these notations we now proceed to the statement of our standing assumptions.

\medskip

Analysing Theorem 2.2 of \cite{armen}, one can easily verify that following assumption is   essential.

\begin{assumption}\label{ass11}
Let $\mu\in \CP(E)$ be a decomposable measure with decomposition $\{F_n,G^n,l_n\}_{n=0}^\infty$.
We assume that for any $n\in\NN$ there exists a positive function  $\rho_n:G^n\times G_n\to \RR^+_0$ such that
	$\mu_{{G^n}}$--a.s.\ we have for all $U\in\CB(G_n)$
	$$
	l_n(\by,U)=\int_U \rho_n(\by,x)\, dx.
	$$
\end{assumption}

Assumption \ref{ass11} is often difficult to verify. Hence we formulate the next assumption which  is more stronger but easier to check than the above.
In fact, we prove in Lemma \ref{abscont} that the following assumption, i.e.\ Assumption \ref{ass1},  implies Assumption \ref{ass11}.

\begin{assumption}\label{ass1}
	Let  $\mu\in \CP(E)$ be a decomposable measure with decomposition
	$$\{F_n,G^n,l_n\}_{n=0}^\infty$$
	such that
	$\mu_{G_n}$ is  absolutely continuous with respect
	to the Lebesgue measure $\leb_{G_n}$.
\end{assumption}

\del{In
Lemma \ref{decompoall} we will show that if we have found one decomposition based on a Schauder basis of $\mu$ satisfying Assumption \ref{ass1},  then
for any finite dimensional subspace $F$ the measure $\mu_F$ will be absolutely continuous with respect to the Lebesgue measure $\leb(F)$.
It follows that
all other decomposition
based on a Schauder basis will also satisfy Assumption \ref{ass1}. 
}
Our third standing set of conditions is given in the following next lines.

\begin{assumption}\label{ass2}
	Let $\mu\in \CP(E)$ and let
	$\{F_n,G^n,l_n\}_{n=0}^\infty$ be a decomposition of $\mu$.
There exists a point
		\DEQSZ\label{yy}
		Y=\sum_{j=1}^\infty y_je_j\in E,
		\EEQSZ
		such that
	\begin{enumerate}
		\item
		for any $n\in\NN$ and $\delta>0$   
		$$
		\mu_{G_n}( B_{G_n}(\pi_{G_n}Y,\delta)\footnote{For a Banach space $E$ we denote by $B_E(y,\delta)$ the ball centered at  $y$ with radius $\delta$, i.e.\ $B_E(y,\delta)=\{ x\in E:|x-y|_E\le \delta\}$.})>0. 
		$$
		\item        
{for all numbers $N\in\NN$ there exists a $R_N>0$}
		 such that for all
		$x_0\in B_{G_N} (\pi_{G_N}Y,R)$,   and all $\ep>0$, there exists a $\delta>0$ such that
		$$
		\mu\lk( \lk\{ x\in E 
		\mid |x-x_0|_{E}\le \ep\rk\}\rk)\ge \delta.
		$$
	\end{enumerate}
\end{assumption}

In order to clarify the role of the above assumption we shall introduce the following definition.
 \begin{definition}
We call a set $A\in \CB^\mu(E)$ a finite zero one $\mu$--set if and only if for all $n\in\NN$
$$
\mu_{G^ n}\lk( \lk\{ y\in G^n: \mu_n(A_n(y))=0\mbox{ or } 1 \rk\}\rk)=1,
$$
where $A_n(y)=A_{(F_1\oplus\cdots\oplus F_n,G^n)}(y)$.

\end{definition}

Let $F_\infty = \cup_{n\in\NN} \{ F_0+F_1+\cdots + F_n\}$. Now, let us present the generalization of Theorem 4 in \cite{dineen}, respective
\cite[Theorem 1.6]{armen}], whose proof requires that the measure $\mu$  is decomposable and has a finite second moment (see \cite[Property (P), page 402]{armen} for the precise statement).

  \begin{theorem}\label{analyticnull}
Let $f:X\to\RR$ be an analytic function and let $\mu\in\CP(X)$ be a decomposable measure with density satisfying Assumption \ref{ass1} and  Assumption \ref{ass2}. Let $\CN_f\subset E$  be defined by  $$\CN_f:=\{ x\in E: f(x)=0\}.$$ Then,  we have
$$ \mu(\CN_f)=0 \,\mbox{or}\, 1.
$$
Furthermore, if $f$ is not identical zero, then $\mu(\CN_f)=0$.
\end{theorem}

\begin{proof}[Proof of Theorem \ref{analyticnull}:]
Let $\CN_f:=\{x\in G:f(x)=0\}$. Since $f$ is analytic, for all $n\in\NN$  for any  $y\in G^ n$   the function $f_y(x):=f(y+x)$ is also analytic.
Therefore, 
either $\leb_{G_n}(\CN_f^ n(y))=0$ or
$\leb_{G_n}(E\setminus \CN_f^ n(y))=0$, where $\CN_f^ n(y)=\{ x\in G_n: x+y\in \CN_f\}$.
Thus, $\CN_f$ is a finite zero--one $\mu$ set, and there exists a set $\tilde {\CN}_f\in\CB(E)$ such that
$\tilde {\CN}_f+F_{(\infty)}=\tilde {\CN}_f$ and $\mu(\tilde {\CN}_f)=\mu(\CN_f)$.

To prove the second part we assume that $f\not \equiv0$ and we shall show that  $\mu(\CN_f^c)>0$.
For this purpose let $n\in \mathbb{N}$ be fixed and  set  $f_n := f_{\lvert_{G_n}}$ and $Y_n:=\pi_{G_n}Y$ where $Y$ is the point from Assumption \ref{ass2}. Observe that $f_n$ is analytic and thus
$$\leb_{G_n}\biggl(G_n\setminus\left\{x\in G_n: f(x)\neq 0    \right\}\biggr)=0 .$$
We shall now distinguish two cases: $f_n(Y_n)\neq 0$ and $f_n(Y_n)=0$. For the first case, \textit{i.e.}, $f_n(Y_n)\neq 0$ we observe that by the continuity of $f_n$ there exists a number $\delta>0$ such that $f(x)\neq 0$ for all $x \in B_E(Y_n, \delta)$  from which along with item (2) of Assumption \ref{ass2} we easily conclude that $\mu(\CN_f^c)>0$.
To treat the second case, \textit{i.e}, $f_n(Y_n)=0$, we first notice that, since $f_n$ is analytic, we have
$$\leb_{G_n}\biggl(\left\{x\in G_n: f(x)= 0    \right\}\biggr)=0,$$
which implies that for any $\ep>0$ one can find $x_0\in G_n$ such that $\lvert x_0 -Y_n \rvert_E \le \ep$ and $f(x_0)\neq 0$. Since $f$ is continuous we can find a number $\delta>0$ such that $f(x)\neq 0$ for all $x\in B_E(x_0, \delta)$. Item (2) of Assumption \ref{ass2} with $\ep=\tfrac R2$ yields that $\mu(B_E(x_0,\delta)>0)$ from which it easily follows that $\mu(\CN_f^c)>0$.
\end{proof}

The above theorem will, as in \cite[Theorem 2.2]{armen}, be used to prove the existence of the density of law of the finite projection on finite dimensional space of the solution of a stochastic evolution equation driven by L\'evy noise and fractional Brownian motion.

\section{The main result}
In this section we consider an abstract stochastic evolution equation in a separable Banach space $E$
\DEQSZ\label{the-eq}\lk\{ \barray
du(t) + {\mL}u(t)\, dt + B(u(t), u(t)) &=& \dot{\Xi}(t),\quad t>0,
\\ u(0)&=& u_0\in \mathscr{H}.\earray \rk.
\EEQSZ
where the driving noise $\Xi$  is either a L\'evy process or a fractional Brownian motion, $\mL:D(\mL)\to \mathscr{H}$  and $B:\mathscr{H}\times \mathscr{H} \to \mathscr{H}$  is a densely defined bilinear operator taking values in $\mathscr{H}$. 
We assume that the above equation is uniquely solvable in $\mathscr{H}$ and we denote the solution starting from $u_0\in \mathscr{H}$ at time $t=0$ by $\{ u(t,u_0):t\ge 0\}$.
\del{ Let $(\CP_{t})_{t\ge0}$ be the Markovian semigroup associated to $u$ defined
by
$$
\lk( \CP_t f\rk) (y) := \EE \lk[ f(u(t,y))\rk],\,\,\, t\ge 0, \,\,\, y\in H,\,\,\, f\in \mathcal{B}_b(H).
$$}

In order to formulate the main result of this section we need to introduce few concepts from the control theory. For this aim, let $U\subset \mathscr{H}$  be a separable Banach space $r\ge1$ be fixed number and let us consider the following control problem
\begin{equation}\label{the-controll}
\hspace{2cm}\lk\{ \barray
du(t) + \mL u(t)\, dt + B(u(t),u(t)) &=& v(t),\quad t>0,
\\ u(0) &=& u_0\in \mathscr{H},
\earray\rk. 
\end{equation}
where $v\in L^r (0,T;U)$ is the control and $U$ 
is the control space (the trajectories of our noise will be basically belong to  $L^r(0,T; U)$ ). For a fixed time $T>0$ we denote by
\DEQSZ 
\label{mmm1}
\CR_T:\mathscr{H}\times L^r (0,T;U) &\to& \mathscr{H}
\EEQSZ 
the so called solution operator that takes each function $g\in L^r (0,T;U)$ and initial condition $u_0\in \mathscr{H}$ to the solution $u(T,u_0)$ of the system
\eqref{the-controll}.

\begin{definition}
A system is controllable in time $T>0$ for a finite dimensional subspace $F\subset \mathscr{H}$ if and only if
$$
\pi_F\CR_T(u_0, L^r (0,T;U))\supset F
$$
for any $u_0\in \mathscr{H}$.
\end{definition}

\begin{definition}
A system is solidly controllable in time $T>0$ for a finite dimensional subspace $F\subset \mathscr{H}$, if and only if
for any $R>0$ and any $u_0\in \mathscr{H}$, there exists an $\ep>0$ and a compact set $K_\ep\subset L^r (0,T;U)$
such that for any function $\Phi:K_\ep\to F$ satisfying
$$
\sup_{x\in K_\ep} |\Phi(x)-\pi_F\CR_T(u_0,x)|_F\le \ep,
$$
we have
$$
\Phi(K_\ep)\supset B_ F(R).
$$
\end{definition}

With this preliminary works the following general result can be shown.

\begin{theorem}\label{mainth}
{Let $E$ be a separable Banach space with Schauder basis $\{e_n: n \in \mathbb{N}\}$. Let $F$ be a finite dimensional subspace of $\mathscr{H}$. We assume that the embedding  $L^r(0,T;U)\hookrightarrow E$ is continuous, $\{e_n: n \in \mathbb{N}\}$ is also a Schauder basis in $L^r(0,T;U)$, and the law $\mu$ of the noise
$\dot{\Xi}$ on $E$ is decomposable on $E$ with the decomposition $\{F_n, G^n, l_n  \}_{n=0}^\infty$, where
notation used in \eqref{def:F_n} and \eqref{def:G_n} is enforced,  
satisfying Assumptions \ref{ass1} and \ref{ass2}.  For a fixed number $T>0$ we also assume that
\begin{enumerate}[label=(\textbf{A}\arabic{*})]
  \item \label{A1}the solution operator  $\CR_T$ defined in \eqref{mmm1}  which is generated by the system \eqref{the-controll} is analytic, 
  \item \label{A2} and  for any finite dimensional space $F\subset \mathscr{H}$, the system \eqref{the-controll} is solidly controllable in time $T$ for the finite dimensional space  $F$.
\end{enumerate}
Then, for any $u_0\in \mathscr{H}$ and for any finite dimensional subspace $F\subset \mathscr{H}$
there exists a density function $\rho:F\to\RR^+_0$ such that 
$$
\EE 1_\CO (\pi _F u(T,u_0)) = \int_\CO \rho(x)\, \leb_F(dx).
$$ }
\end{theorem}

\begin{proof}
Let us fix a finite dimensional subspace $F$ of $\mathscr{H}$ and consider the operator
$$
f: \mathscr{H}\times X\ni (u_0,\xi) \mapsto \pi_F \CR_T (u_0,\xi) \in F,
$$
where $X=L^ r(0,T;U)$, $u$ solves equation \eqref{the-eq} and $\CR_T$ is defined in \eqref{mmm1}.

{The proof of our theorem will follow from the applicability of \cite[Theorem 2.2]{armen}. Thus we just need to check that all the assumptions of \cite[Theorem 2.2]{armen} are all satisfied. For this aim it is sufficient to prove the two claims below.}

\begin{enumerate}[leftmargin=0cm,itemindent=.5cm,labelwidth=\itemindent,labelsep=0cm,align=left, label={\textbf{Claim}\arabic{*}}]
\item \
\label{Claim-1}.
There exists a finite dimensional subspace
$G_m$ of $X$
such that
for any $u_0\in \mathscr{H}$, there exists a ball $B_0\subset G_m$ and a ball  $B_F\subset F$
such that
$$
f(u_0,B_0)\supset B_F.
$$

\medskip

To prove this claim we fix a large number $R>0$ such that $u_0\in B_\mathscr{H}(R)$. 
By the definition of solidly controllability, we know that there exists an $\ep>0$ and a compact set $K_\ep\subset \mathscr{H}$ such that,  any function $\Phi:K_\ep\to \mathscr{H}$ satisfying
$$
\sup_{y\in K_\ep } | \Phi(y)-\pi_F\CR_T(u_0,y)|_F\le \ep,
$$
satisfies
$$
\Phi(K_\ep )\supset \{ y\in F: |y|_F\le R\}. 
$$
Fix $u_0\in B_\mathscr{H}(R)$, $\ep>0$ and the corresponding compact set $K_\ep$.
Since the operator
$$
\CR_T(u_0,\cdot):X\to \mathscr{H}
$$
is continuous, it is uniformly continuous on $K_\ep$, and, hence,
 there exists a $\delta_0>0$ such that
$$\lk|\CR_T(u_0,y_1)- \CR_T(u_0,y_2)\rk|_\mathscr{H}
\le \ep,\quad \forall \, y_1,y_2\in K_\ep \mbox{ with } |y_1-y_2|\le \delta_0\, .
$$
Since the function system $\{ e_n:n\in\NN\}$ is a Schauder basis of $X$, it follows that $\cup_{m\in\NN}F_m$ is a dense subset in $X$.
In particular, since $K_\ep$ is compact, for any $\delta>0$, there exists  a number $m$ such that
$$
\sup_{y\in K_\ep} \lk\| y-\pi_{G_m}y\rk\|_{X}\le \delta.
$$
Let $m\in\NN$ be sufficiently  large such that
$$
\sup_{y\in K_\ep} \lk\| y-\pi_{G_m}y\rk\|_{X}\le \delta_0,
$$
Let us define
$$
\Phi:K_\ep\to \mathscr{H}
$$
by
$$
\Phi(y) = \pi_F(\CR_T(u_0,\pi_{G_m}y)).
$$
From the consideration above, it follows that
$$
\sup_{y\in K_\ep} |\Phi(y)-\pi_F\CR_T(u_0,y)|_F\le \ep.
$$
Hence, by the solid controllability
$$
\Phi(K_\ep)\subset  \{ y\in F: |y|_F\le R\}.
$$
In particular, since $\pi_{G_m}K_\ep $ is a bounded set of $G_m$,  there exists a number $R_1>0$ such that $\{y\in G_m:|y|\le R_1\} \supset \pi_{G_m}K_\ep$. Setting $B_F:=
 \{ y\in F: |y|_F\le R\}$ and $B_1:= \{y\in G_m:|y|\le R_1\}$ we have 
$$
\CR_T( u_0,B_1 ) \supset B_F,
$$
which proves \ref{Claim-1}.

\medskip

\item\label{Claim-2}. \ The measure $\mu$ on $E$ satisfies Assumption \ref{ass11}.

\medskip

Claim 2 is easy to prove. Thanks to Lemma \ref{abscont} the measure satisfies Assumption \ref{ass2}, which is equivalent to Claim 2.
\end{enumerate}
\end{proof}


\section{Application to the 2D stochastic Navier-Stokes}

Throughout this section $\mo$ denotes the 2D torus, $\el^p(\mo)$ and $W^{m,p}$
will respectively
denote the usual Lebesgue space of $p$-integrable functions  and  Sobolev spaces. The symbol $B^s_{p,p}(\INT):=B^s_{p,p}(\INT;\mathbb{R})$ is the Besov spaces of all $\mathbb{R}$-valued functions defined on the interval $\INT$.

Let $\mathcal{V}$ be the set of periodic, divergence free and infinitely differentiable function with zero mean.
In what follows, we denote by $\h$ and $\ve$ the closures of $\mathcal{V}$ in $\el^2(\mo%
)$ and $W^{1,2}(\mo)$, respectively. We endow the space $\h$ with the $L^2$-scalar product denoted by $(\cdot, \cdot)$ and the usual $L^2$-norm denoted by $\lvert \cdot \rvert$. The space $\ve$ is equipped with the gradient norm $\lvert \nabla \cdot \rvert$. We also set $$ D(\rA)= [\h^2(\mo)]^2\cap \ve,\quad  \rA \bv= -\Pi \Delta \bv, \quad \bv \in D(\rA), $$  where   $\Pi$ is the orthogonal projection from $\el^2(\mo)$ onto $\h$.
It is well-known that the Stokes operator $\rA$ is positive self-adjoint with compact resolvent and its eigenfunctions
$\{e_1, e_2,\ldots\}$, with eigenvalues $0<\lambda_1\le \lambda_2\le \ldots$, form an orthonormal basis of $\h$.
It is also well-known that $\ve=D(\rA^\frac12)$, see \cite[Appendix A.1 of Chapter II]{Temam-RB}.
Furthermore, we see from \cite[Chapter II, Section 1.2]{Temam_2001} and \cite[Appendix A.3 of Chapter II]{Temam-RB} that
one can define a continuous bilinear  map $\bbop$ from $\ve \times \ve$ with
values in $\ve^\ast$ such that
\begin{align}
\langle
& \bbop(\bu,\bv),\bw\rangle=\int_{\mo} [\bu(z)\cdot \nabla \bv(z)]\cdot \bw(z) dz \quad \text{ for any }  \bu, \bv, \bw \in \ve,\label{DEF-B1}\\
& \langle \bbop(\bu,\bv),\bv\rangle=0, \quad \text{ for any } \bu, \bv \in \ve,\\
& \vert \langle \bbop(\bu,\bv), \bw \rangle \vert\le C_0 \Vert \bu \Vert _{\el^4} \Vert \bv \Vert_{\el^4} \Vert \bw \Vert, \text{ for } \bu,\, \bv \in \el^4(\mo),\, \bw \in \ve.
\end{align}


With all these notations the Navier-Stokes equations \eqref{eqn:4.1} can be written in the abstract form
\DEQSZ\label{navstab} 
\lk\{ \barray
\frac{d u(t)}{d t}+\kappa \rA u(t) +\bbop(u(t),u(t)) =\dot{\zed}(t),\phantom{\Big|} \\
u(0)=u_0 \in \h, \earray\rk.
\EEQSZ 
where for the sake of simplicity we assume that  $\Pi \dot{\zed}= \dot{\zed}$. The positive number $\kappa>0$ denotes the viscosity.
 Before characterizing the noise entering our system,  we introduce the
trigonometric basis in $\mathscr{H}$ by elements in $\mathbb{Z}$. Namely, we write $j-(j_1,j_2)\subset \mathbb{Z}^2$ and set
\DEQS
e_j(x) &=& \sin( jx)j^ \perp \quad \mbox{ for $j_1>0$ or $j_1=0, j_2>0$,}
\\
e_j(x) &=& \cos( jx)j^ \perp \quad \mbox{ for $j_1,0$ or $j_1=0, j_2,0$,}
\\
e_0^ 1 (x) &=& (1,0),\quad e_0^ 2 (x) = (0,1),
\EEQS
where $j^ \perp=(-j_2,j_1)$. The family $\mathcal{E}=\{ e^ j_0,e_j,i=1,2,j\in\mathbb{Z}\setminus \{0\}\}$ is a complete set of eigenfunctions for the Stokes operator which forms an orthonormal basis in $\mathscr{H}$.

For any symmetric set $\CK\subset \mathbb{Z}^ 2 $ containing $(0,0)$ we write $\CK_0=\CK$ and define $\CK^ i$ with $i\ge 1$ as the union for $\CK^ {i-1}$ and the family of vectors $l\in\mathbb{Z}^ 2 $ for which there are $m,n\in \CK^ {i-1}$ such that $ l=m+n$, $|m|\not =|n|$, and $|m\wedge n\not=0$, where $m\wedge n=m_1n_2-m_2n_1$.

\begin{definition} A symmetric subset $\CK\subset \mZ^ 2 $ containing $(0,0)$ is saturating, if and only if $\cup_{i\in\NN} \CK^ {i-1}=\mZ^ 2$.
\end{definition}
%
Throughout we set $d={\dim{\CK}}$  and denote by $\h_d$ the finite dimensional subspace of $\h$ spanned by the eigenvectors $\{e_j; j\in \CK\}$. The driving noise is either
\begin{equation}\label{Noise-Levy}
\hspace{3cm}\zed(t) = \sum_{j\in \CK} e_jl_j(t),\quad
t\ge 0,
\end{equation}
where $\{ l_j:j\in\CK\}$ is a family of  identical distributed and mutual independent  L\'evy processes with L\'evy measure $\nu_j$ over a probability space $(\Omega,\CF,\PP)$, or
\begin{equation}\label{Noise-FracBM}
\hspace{3cm}\zed(t) = \sum_{j\in \CK} e_j\beta ^H_j(t),\quad
t\ge 0,
\end{equation}
where  $\{ \beta^H_j:j\in\CK\}$ is a family  of  identical distributed and mutual independent fractal Brownian motions with Hurst parameter $H\in(\tfrac 12,1)$ over a probability space $(\Omega,\CF,\PP)$.
{The existence of a unique solution $u=\{u(t):t\ge 0\}$ to \eqref{navstab} follows from the results in \cite{ehz} for example for the case of pure jump L\'evy noise, and from \cite{nsfractal} for case of fractional Brownian motion  perturbation.}
%
%

We can now state the main results of this section. We start with the following theorem which treats the case of Navier-Stokes equations driven by L\'evy noise.
\begin{theorem}\label{coreins}
Let $\CK$ be a saturating set and assume that the noise $\zed$ entering the system \eqref{the-eq} is defined by \eqref{Noise-Levy}.
We also assume that the L\'evy measures $\nu_j$, $j=1,\ldots ,d$, are symmetric and equivalent to the Lebesgue measure on  $\RR\setminus \{0\}$ and satisfies
\DEQSZ \label{p-integrable}
\int_{|z|\le 1} |z|^ p\nu_j(dz)<\infty,
\EEQSZ 
for some $p \in(1,2)$. In addition, we assume that there exists a number $\alpha\in (0,2]$ such that
$$ \nu_j(\RR\setminus[-\ep,\ep])\sim \ep^ {-\alpha}l(\ep) \quad \mbox{as} \,\ep\to 0,
$$
for some slow varying function $l$. 
Let $u=\{u(t,u_0):t\ge 0, u_0\in \mathscr{H}\}$ be the unique solution of system  \eqref{the-eq}. Then
for any finite dimension subspace $F\subset \mathscr{H}$,  for all initial conditions $u_0\in \mathscr{H}$,
there exists a density function $\rho_{u_0}:F\to\RR^+_0$ such that
$$
\EE 1_\CO (\pi_F u(T,u_0)) = \int_\CO \rho_{u_0}(x)\, \leb_F(dx).
$$
In addition, for any sequence $\{u_n:n\in\NN\}$ with $u_n\to u_0\in \mathscr{H}$ as $n\to\infty$, we have
$$
\int_F|\rho_{u_0}(x)-\rho_{u_n}(x)|dx 
{\longrightarrow} 0 \mbox{ as }{n\to \infty}.
$$	
\end{theorem}

\begin{proof}
{For simplicity, let us assume $T=1$. As in the previous section we  consider map
\DEQSZ 
\label{navmmm1}
\CR_T:\mathscr{H}\times L^2 (0,T;\mathscr{H}_d) &\to& \mathscr{H}
\EEQSZ 
which is the solution operator that takes each function $g\in L^2(0,T;\mathscr{H}_d)$ and initial condition $u_0\in \mathscr{H}$ to the solution $u(T,u_0)$ of the control system
\eqref{the-controll} associated to the Navier-Stokes equations. It is proved in \cite[Proposition A.2]{armen}, see also \cite{kuksin},
that the operator $\CR_T$  is analytic. It is also known from \cite[Proposition A.5 ]{armen}, see also \cite{as}, that the system \eqref{the-controll} for the Navier-Stokes is solidly controllable in time $T$ for any finite dimensional space $F\subset \mathscr{H}$.  Hereafter we respectively identify $\mathscr{H}_d$ and $F$ to $\mathbb{R}^d$ and $\mathbb{R}^{\dim F }$.
Let $p\in (1,2)$ such that \eqref{p-integrable} is satisfied.
Let  $p'$ be the conjugate exponent to $p$ and $s<\tfrac 1p-1$. For each $j \in \CK$ let $\xi_j$ be the map defined by
\DEQS
\xi_j:B_{p',p'}^ {s}(\INT,\RR)\ni \phi \mapsto \xi_j(\phi)=\int_0^ 1 \phi(\tau)\, dl_j(\tau)\in L^ 0(\Omega;\RR),
\EEQS
and  $\mu_j$ be the cylindrical measure on $B_{p,p}^s(\INT,\RR)$ defined by
$$
\mu_j\lk( \lk\{ x\in B_{p,p}^ {s}([0,1]): (x(\phi_1),\ldots,x(\phi_n))\in C\rk\}\rk) := \PP\lk(  (\xi(\phi_1),\ldots,\xi(\phi_n))\in C \rk), \, C\in \CB(\RR^n),
$$%
where $n\in\NN$, $\phi_1,\ldots,\phi_n\in \CS(\RR )$. In Proposition \ref{finitenesslevy} we show that the cylindrical measure is actually a Radon probability measure on $B_{p,p}^ {s}([0,1])$.

 From the results of Section \ref{haar} we infer that the  probability measure $\mu_j$ on $B_{p,p}^s(\INT,\RR)$  is decomposable with decomposition
 	$\{F_n,G^n,l_n\}_{n=0}^\infty$, where $F_n$ and $l_n$ are respectively defined by $F_0=V_0$, $F_n=W_{n}$, $n\ge 2$, where  $V_0$ and $W_n$
are defined in  \eqref{exactspaceslevy} and the existence of $l_n$ is given by Lemma \ref{lemmaa3}.
With $F_n$ at hand the space $G^n$ is defined as in Definition \ref{decomposkerneldef}. Moreover, we infer from Lemma \ref{lemma11}  that for each $j$
the probability measure $\mu_j$ satisfies Assumptions \ref{ass1} and \ref{ass2}. With these observation in mind, it is not difficult to check that the product measure $\mu=\otimes_{j\in\CK}\mu_j$ satisfies Assumptions
\ref{ass1} and
\ref{ass2} on the Banach space $E:=B_{p,p}^s(\INT,\RR^d)$ where $d=\dim(\CK)$. Now, the proof of the theorem easily follows from an application of
Theorem \ref{mainth}.}
\end{proof}
We now proceed to the statement and the proof of the above theorem when the noise entering the system is a fractional Brownian motion given by \eqref{Noise-FracBM}.
\begin{theorem}
Let $\CK$ be a saturating set and assume that the  noise $\zed$ is a fractional Brownian motion defined by \eqref{Noise-FracBM} with Hurst parameter $H\in (\tfrac 12,1)$.
Let $u=\{u(t,u_0):t\ge 0, u_0\in H\}$ be the unique solution of system  \eqref{the-eq} with initial condition $u_0$. Then, for any finite dimensional space $F\subset \mathscr{H}$
and initial condition $u_0\in \h$,
there exists a density function $\rho_{u_0}:F\to\RR^+_0$ such that
$$
\EE 1_\CO (\pi_F u(T,u_0)) = \int_\CO \rho_{u_0}(x)\, \leb_F(dx).
$$
In addition, for any sequence $\{u_n:n\in\NN\}$ with $u_n\to u_0\in \h$ as $n\to\infty$, we have
$$
\int_F|\rho_{u_0}(x)-\rho_{u_n}(x)|dx 
{\longrightarrow} 0 \mbox{ as }{n\to \infty}.
$$
\end{theorem}
\begin{proof}
{
Let  $\CR_T: \h\times L^2 (0,T;\mathscr{H}_d)\to \h$ be  the solution operator defined by \eqref{navmmm1} in the proof of Theorem \ref{coreins}.
It satisfies the properties enumerated in the proof of Theorem \ref{coreins}. Hereafter we respectively identify $\mathscr{H}_d$ and $F$ to $\mathbb{R}^d$ and $\mathbb{R}^{\dim F }$.
Let $s\in (-\tfrac 12,\, H-1)$.
For each $j \in \CK$ let $\xi_j$ be the map defined by
  \DEQS
\xi_j:B_{2,2}^ {-s}(\INT)\ni \phi \mapsto \xi_j(\phi)=\int_0^ 1 \phi(\tau)\, d\beta ^H_j(\tau) \in L^ 2(\Omega;\RR),
\EEQS
and  $\mu_j$ be the cylindrical measure on $B_{2,2}^s(\INT,\RR)$ defined by
$$
\mu_j\lk( \lk\{ x\in B_{2,2}^ {-s}([0,1]): (x(\phi_1),\ldots,x(\phi_n))\in C\rk\}\rk) := \PP\lk(  (\xi(\phi_1),\ldots,\xi(\phi_n))\in C \rk),\, C\in\CB(\RR^n),
$$%
where $n\in\NN$, $\phi_1,\ldots,\phi_n\in \CS(\RR )$. 
 From the results of Section \ref{haartrue} we infer that the cylindrical measure $\mu_j$ on $B_{2,2}^s(\INT,\RR)$ is actually a probability measure and is decomposable with decomposition
 	$\{F_n,G^n,l_n\}_{n=0}^\infty$, where $F_n$ and $l_n$ are respectively defined by   $F_0=V_0$, $F_n=W_{n}$, $n\ge 2$,
 where  $V_0$ and $W_n$
are defined in  \eqref{exactspacesbr}.
With $F_n$ in mind we define $G^n$ as in Definition \ref{decomposkerneldef}. We also infer from Lemma \ref{lemma11}  that for each $j$
the probability measure $\mu_j$ satisfies Assumptions \ref{ass1} and \ref{ass2}. We now easily complete the proof by using a similar argument as in the proof of Theorem \ref{coreins}.}
\end{proof}


\appendix

\section{The L\'evy Noise and its Wavelet Expansion}
\label{haar}

In this section we assume that we are given a real-valued \levy process $\ell$ with $\sigma$-additive L\'evy measure $\nu$ on $\mathbb{R}\backslash \{0\}$  satisfying \eqref{p-integrable}, \textit{i.e.}
\DEQS 
\int_{|z|\le 1} |z|^ p\nu(dz)<\infty,
\EEQS 
for some $p \in(1,2)$.	
	  Our aim is to investigate the expansion of the process $\ell$ in terms of Debauchies wavelets of order $k$.  To keep this section and the article short we refer to the reader for the technical jargon about wavelets to  \cite{daub} or  \cite{triebel_fkt}.

We start introducing  the Daubechies  wavelets, see for \textit{e.g.} \cite{daub}. For such aim we fix $u>0$  and consider the Debauchies wavelets  $\psi$ having continuous bounded derivatives up to order $k$. It is known, see for \textit{e.g} \cite{daub}, that to $\psi$ we can associate scaling function denoted by $\phi$. With these in mind, %
the system of wavelets is given by
$$\psi_{j,k}:=2^{-\frac j2}\psi(2 ^jt+k)\mbox{ and }  \phi_{j,k}:=2^{-\frac j2}\phi(2 ^jt+k), \quad \, j\in\NN,\, k\in J_j,
$$
where $J^\psi _j=\{ k\in\NN: \mbox{supp}(\psi_{j,k})\cap I\not=\emptyset\}$, $J^\phi _j=\{ k\in\NN: \mbox{supp}(\phi_{j,k})\cap I\not=\emptyset\}$. The corresponding 
 multiresolution analysis is defined  by
\DEQSZ\label{exactspaceslevy}
V_n:= \spn\{ \phi_{j,k}: j=1,\ldots,n,\,\, k\in J^ \phi_j\},
\quad
W_n:= \spn\{ \psi_{n,k}:  k\in J^ \psi_n\}.
\EEQSZ
For detail on the properties of the wavelet basis we refer to  \cite[Theorem 1.58]{triebel_fkt} or to \cite{kahane}. Note that  for $s\in\RR$  the Daubechies wavelets of order $k$, with  $k>\max( s,(1-\frac 1p)_+ -s) $,  form an unconditional basis of $B^s_{p,p}(\INT)$. In particular,  for each element $f\in B^s_{p,p}(\INT)$ there exists a unique sequence
 $$\{ \lambda_{j,k}: j\in\NN, k\in J_j^ \psi\}$$ such that $f$ 
 can be written as
\DEQSZ 
\label{Wave-exp}
f=\sum_{j\in\NN} \sum_{k\in J^ \psi_j} \lambda_{j,k} \psi_{j,k} + \lambda _0\phi .
\EEQSZ 


%
Note that since we are  considering the process on the time interval $[0,1]$, we only need to sum over $J_j^\psi$. We also note that $|J_j^\psi|\sim 2^ j$.
\del{Fix in the following  $s<\frac 1p-1$, and $u>\max( s,\frac 2p+\frac 12 -s) $. Let $E=B^s_{p,p}(\INT)$. For later on, let us define the following subspaces of $E$: 
$F_0:=V_1$, $F_n=W_{n}$ and $G_n:= F_0\oplus F_1\oplus\ldots \oplus F_n$.
Let us denote the projection of of a function $f$ belonging to $ E$  onto $G_n$ by $\pi_{G_n}$
and  onto $W_n$ by $\pi_{W_n}$.
In addition, we have the following well known decomposition $G_{n+1}= W_{n}\oplus G_n=F_n\oplus G_n$.}

\medskip

\del{By  means of the multresolution analysis we can construct from  each $E $--valued random variable a martingale.
Therefore, let $X$ be an $E$--valued random variable $X$, and let us define the sequence $(M_n:n\in\NN)$  by
$$
M_n:=\EE[ X\mid \CF_n],\quad n\ge 1.
$$
Observe, $\{ M_n:n\in\NN\}$ is a martingale and $M_n= \pi_{G_n} X$, and $M_{n+1}-M_n=\pi_{W_n}X$.
The second follows, since $G_{n+1}= W_n\oplus G_n=F_n\oplus G_n$.

\medskip
}

{In the next paragraph, we will construct the probability measure induced  by a L\'evy process which will be represented as an integral with respect to a Poisson random measure. This representation is motivated  in one hand by the fact that the use of Poisson random measure simplifies many calculation. In other hand  the Poisson random measure framework  seems more general.}
We refer to \cite{apple-int}, \cite[Chapters 6-8]{PESZAT+ZABZCYK} and \cite[Chapter 4]{sato} for a precise  connection between Poisson random measures and L\'evy processes and stochastic integration with respect to them.

\medskip

Over a probability space $\mathfrak{A}=(\Omega,\CF,\PP)$, we consider a time homogenous Poisson random measure $\eta$  on $\RR$  with symmetric  intensity measure $\nu$ as above.

\begin{proposition}
 The Poisson random measure $\eta$ over a probability space $(\Omega,\CF,\PP)$ induces  a Radon probability  measure $\mu$  on $B_{p,p}^ s(\INT )$.
\end{proposition}
\begin{proof}
 We will start the proof with removing jumps of size bigger than $\ep\in(0,1]$ and let $\epsilon$ converges to 0.  For this purpose we take an arbitrary constant $\ep\in(0,1]$ and define a Poisson random measure $\eta_\ep$ by
\DEQS
\eta_\ep:\CB(\RR)\times\CB(\INT) &\to& \bar \NN ,
\\
(A\times I) &\mapsto& \eta(A\cap (\RR\setminus [-\ep,\ep])\times I).
\EEQS
The family $\{ \eta_\ep:\ep\in(0,1]\}$ induces a family of cylindrical measures on $C_b(\INT)'$. 
Here, it is important that the to Poisson random measure $\eta_\ep$ corresponding  the L\'evy process can be written as a sum over finitely many jumps at certain, possibly random, jump times.
To be more precise, let $\nu_\ep$ be defined by $\nu_\ep (A)=\nu(A\cap (\RR\setminus [-\ep,\ep]))$,
$\rho_\ep=\nu(\RR\setminus [-\ep,\ep])$, let $N_\epsilon$ be a Poisson distributed random variable with parameter $\rho_\ep$, $\{ \tau^\ep_n:n=1,\ldots, N\}$ be a family of independent uniform distributed random variables on $[0,1]$, and $\{ Y_n:n=1,\ldots, N\}$ be a family of independent, $\nu_\ep/\rho_\ep$ distributed random variables. Denoting $\delta_x$ the Dirac distribution concentrated at $x$,
the Poisson random measure $\eta_\ep$ can be written as
$$
\eta_\ep (A\times I ) = \sum_{n=1}^ N \delta_{\tau_n}(I) \delta _{Y_n}(A)
$$
and for any $f\in C_b(\INT)$ the mapping
$$
\xi_\ep (f) := \int_0^ 1\int_\RR f(s) z\eta_\ep(dz,ds)=\sum_{n=1}^ N f(\tau_n)Y_n
$$
is well defined.

\medskip

Let us define the random variables
$$
\zeta^ \ep _{j,k} 
{=} \int_0^1 \int_\RR \psi_{j,k}(\tau) \, z\eta_\ep(dz,d\tau),\quad \, j\in\NN,\, k\in J_j,
$$
$$
a_0 ^\ep 
{=} \int_0^1 \int_\RR \phi_{0,0}(\tau) \, z\eta_\ep(dz,d\tau).
$$
Since the mother wavelet $\psi$ and the scaling function {$\phi $}
are  continuous, the families $\{ \zeta^ \ep _{j,k}: j\in\NN,\, k\in J_j\}\cup\{ a_0^\ep\}$
of random variables over $\mathfrak{A}$ are well defined.
In addition, by the definition of $\zeta^\ep$ and $a_0$ and the fact that the multiresolution analysis is a Schauder basis in $B^ {s}_{p,p}(\INT)$, and $\delta\in B^ {s}_{p,p}(\INT)$ (see \cite[Remark 3, p.\ 34]{runst}), we infer that $\xi_\ep$ admits a  wavelet series representation as in \eqref{Wave-exp}.

\medskip

Note that for any $C\in\CB(\RR)$,
$$ \mu_\ep( \{ x\in B_{p,p}^s(\INT): x(\psi_{j,k})\in C\}) = \PP( \zeta^ \ep _{j,k}\in C).
$$
Later on we will need the following proposition which will be proved at the end of the current proof.

\begin{proposition}\label{finitenesslevy}
Let $\nu$ be a L\'evy measure satisfying \eqref{p-integrable} for some   $p\in [1,2)$ and $\ep \in (0,1]$.
Let $$\xi_\ep:= \sum_{j=1}^ \infty   \sum_{k\in J^ \psi_j}  \zeta_{k,j} ^ \ep \psi_{j,k} + a_0^ \ep \phi_{0,0}.$$
Then,
\begin{enumerate}
  \item  for  any $s<\frac 1p-1$, there exists a $C>0$ such that
$$
\EE \lk[ |\xi_\ep|_{B_{p,p}^{s}}^ p \rk]\le C.
$$
  \item  For  any $s<\frac 1p-1$ and $\ep_1,\ep_2\in (0,1]$ we have
$$
\EE \lk[ |\xi_{\ep_1}-\xi_{\ep_2}|_{B_{p,p}^{s}}^ p \rk]\le C\min(\ep_1,\ep_2)^ {2-p}.
$$
\end{enumerate}

\end{proposition}

By the choice of $s$ and $p$, we have $B^ {-s}_{p',p'}(\INT)\hookrightarrow C_b(\INT)$ and $\eta$ is a finite measure.
Secondly, the mappings $\xi_\ep$ induces a the family of cylindrical measures $\mu_\ep$ on $B^s_{p,p}(\INT)$
defined by
$$
\mu_\ep\lk( \lk\{ x\in B_{p,p}^ {s}(\INT): (x(\phi_1),\ldots,x(\phi_n))\in C\rk\}\rk) := \PP\lk(  (\xi_\ep(\phi_1),\ldots,\xi_\ep(\phi_n))\in C \rk),
$$%
$\phi_1,\cdots,\phi_n\in (B_{p,p}^s(\INT))'=B_{p',p'}^{-s}(\INT)$, and $C\in \CB(\RR^n)$.

\medskip

We will now show that  the family of cylindrical measures $\{\mu_\ep:\ep\in (0,1]\}$ has a limit.  In fact,
the family of probability measures $\mu_\ep$   is tight on $B_{p,p}^ s(\INT)$. To show this claim  we fix a constant $s_0\in (s,\frac 1p-1)$. We firstly note  that the embedding $B_{p,p}^ {s_0}(\INT) \hookrightarrow B_{p,p}^ s(\INT)$ is compact. Secondly, the Chebyscheff inequality and Proposition \ref{finitenesslevy} give that  for any $\delta>0$
we can find a compact  $K_\delta:=\{ x\in B_{p,p}^ s(\INT): |x|_{ B_{p,p}^ {s_0}}\le \delta^ {-1/p} \}$ such that
$$
\PP\lk( \xi_\ep\not\in K_\delta\rk) \le \delta\, \EE\lk[  |\xi_\ep |^ p_{ B_{p,p}^ {s_0}}\rk]\le   C \delta.
$$
It follows that the family of probability measures $\{\mu_\ep:\ep\in(0,1]\}$ is tight on $B_{p,p}^ s(\INT)$. It even follows from Proposition \ref{finitenesslevy} that the sequence $\{\mu_{\ep_n}:\ep=\frac 1n\}$ 
forms a Cauchy sequence and the limit $\mu$ is unique.
Therefore, there exists a unique cylindrical measure $\mu$ on $B^ {s}_{p,p}(\INT)$. Since there exists a constant $C>0$ such that for all $\ep>0$
 $\EE\lk[  |\xi_\ep |^ p_{ B_{p,p}^ {s}}\rk]\le C$, it follows from the Lebesgue Dominated Convergence Theorem that $\EE\lk[  |\xi |^ p_{ B_{p,p}^ {s}}\rk]\le C$.
 Hence, $\mu$ is also a Radon probability measure  on $B^ {s}_{p,p}(\INT)$.

\bigskip

Now we shall consider the general case in which $\nu$ is assumed to satisfy  \eqref{p-integrable} for some $p \in (1,2)$. For this purpose we consider the Poisson random measures $\eta_1$ and $\eta_2$ defined by
 $$
\CB([0,\infty))\times \CB(\RR)\ni (I\times A) \mapsto  \eta_1(I\times A) :=\eta( I\times A\cap [-1,1])
$$
and
 $$
\CB([0,\infty))\times \CB(\RR)\ni (I\times A) \mapsto  \eta_2(I\times A) :=\eta( I\times A\cap\RR\setminus  [-1,1]),
$$
respectively. Since $ A\cap [-1,1] \cap A\cap\RR\setminus  [-1,1]=\emptyset$, the Poisson random measures $\eta_1$ and $\eta_2$ are independent. Hence, the two families of coefficients in the wavelet expansion $\eta_1$ and $\eta_2$ are independent too.
 In addition from the first part of the proof
$\eta_1$ induces a Radon probability measure on $B_{p,p}^ s([0,1])$.
Since the process
$$
L_t^ 2:= \int_0^t \int_\RR z\eta_2(dz,ds)
$$
can be written as a finite sum over jumps happen at certain, possibly random, times  within the interval $\INT$, $\dot{L}_t^2$ consist of a sum over finitely many Dirac distributions.
Since any Dirac distributions belong to $B_{p,p}^ s(\INT)$, $\dot{L}_t^2$  is an element of $B_{p,p}^ s(\INT)$ and induces a probability measure on $B_{p,p}^ s(\INT)$.
 Hence, $\eta$ itself induces a Radon probability measure on $B_{p,p}^ s([0,1])$.
\end{proof}
\begin{proof}[Proof of Proposition \ref{finitenesslevy}:]
We recall that
$\int |z|^ p \nu(dz)<\infty$ for some $p\in (1,2)$.
{By the definition of the norm we get 
\DEQS
\EE |\xi_\ep|_{B_{p,p}^{s}}^ p
 & \sim &\EE
\sum_{j=1}^ \infty  2^ {j(s-\frac 1p)p} \sum_{k\in J^ \psi_j} \lk| \zeta_{k,j}^\ep\rk|^ p 2^ {j\frac p2 }
\EEQS
Since
\DEQS
\EE|\zeta^\ep _{j,k}|^ p &\le & C_\nu  \int_0^1  |\psi_{j,k}(s) |^ p \, ds = 2^ {\frac {jp}2}
2 ^ {-j} ,
\EEQS
we infer that there exists a constant $C>0$ such that
\DEQS
\EE |\xi_\ep|_{B_{p,p}^{s}}^ p
& \le C  & \sum_{j=1}^ \infty 2 ^{j(ps-1)}2^ j 2 ^{j(\frac p2-1)} 2^ {j\frac p2 }
 \le C \sum_{j=1}^ \infty 2 ^{j(ps+\frac p2 -1+\frac p2)}
,
 \EEQS
 which is finite for $s<\frac 1p-1$.}

\end{proof}

Let us denote the {Radon probability measure} induced by $\eta$ on $B_{p,p}^s(\INT)$ by $\mu$
and let us  define the mapping
\DEQSZ\label{xidef}
\xi:B_{p',p'}^ {-s}([0,1])\ni \phi \mapsto \xi(\phi)=\int_0^ 1 \int_{\RR}\phi(\tau)\, z\, \eta(dz,d\tau).
\EEQSZ
This mapping is well defined thanks to the above  calculation.

\medskip
We are now interested in the properties of the decomposition of $\mu$ by 
 the multiresolution analysis.
In particular, we will show that for any $n\in\NN$, the probability measure $\mu_{G_n}$ is equivalent to the Lebesgue measure.

We firstly note that since $V_n=W_n\otimes W_{n-1}\otimes \cdots \otimes W_1\otimes V_0$, given the coefficients $\{ \zeta_{j,k}: j=1,\ldots,n, \,\,k\in J^ \psi_j\}\cup\{ a_0\}$, one knows the
coefficient of $\phi_{n+1,k}$.
 For $k\in J_{n+1}^ \phi$  let us denote $\gamma_{n,k}$ the coefficients of $\phi_{n+1,k}$.
In particular,  we have
$$
\gamma_{n,k}:= \int_0^1 \int_\RR \phi_{n,k}(t)z\eta(dz,dt),
$$
 which implies that
$$
\pi_{G_n} \xi = \sum_{k\in J^ \phi_n}\gamma_{n,k} \phi_{n,k}.
$$
Let us now denote  by $\bz^n$ and  $\bg^n$  the random  vectors  $(\zeta_{n,0} ,\zeta_{n,1} ,\ldots ,\zeta_{n,|J^ \phi_n|} )$
and $(\gamma_{n,1} ,\gamma_{n,2} ,\ldots ,\gamma_{n,|J^ \psi_n|} )$, respectively.
Finally, 
for a function  $f:\INT\to\RR$  we write
$$
\xi(f) := \int_0^1 \int_\RR  f(s) z \, \eta(dz,ds).
$$
\begin{lemma}\label{lemaa1}
Let $f:[0,1]\to\RR$ be a mapping such that there exists constants $\delta>0$ and $t_1,t_2\in[0,1]$, $t_1<t_2$  such that
$|f(t)|\ge \delta$ for all $t\in[t_1,t_2]$. Then
\begin{enumerate}
  \item $\supp(\xi(f))=\RR$;
  \item the law of $\xi(f)$ is absolutely continuous with respect to the Lebesgue measure.
\end{enumerate}
\end{lemma}
\begin{proof}
Let us define the following L\'evy measure
$$
\nu_{t_1,t_2}: \CB(\RR)\ni B \mapsto \int_{t_1}^{t_2} \int_\RR 1_{B}(f(t)z) \nu(dz)\, dt.
$$
Then $\xi (f1_{[t_1,t_2]})$ is an infinite divisible random variable, and item (i) follows from  \cite[Corollary 24.4]{sato}. Item (ii) follows from  \cite[Theorem 27.7]{sato}.
\end{proof}

\begin{lemma}\label{exdensity}
For any $n\ge 1$, the measure
\DEQSZ\label{inddensity}
\CB(\RR^{|J^\phi_{n+1}|})\ni U \mapsto \PP\lk( \mathbf{g}^{n+1} 
\in U\rk)
\EEQSZ
is equivalent to the Lebesgue measure on $\RR^{|J^\phi_{n+1}|}$.
\end{lemma}

\begin{proof}
This follows from the fact that for all $k =\min(J^\phi_{n+1}), \ldots, \max(J^\phi_{n+1})-1$, the functions $\phi_{n+1,k}$ and $\phi_{n+1,k+1}$
have disjoint supports. Let us write $\phi_{n+1,k+1}=f_1+f_2$ with $\supp(f_1)\cap \supp(\phi_{n+1,k})
=\emptyset$, $\supp(f_1)$  is an interval $[a,b]$,  $\{ s:f_2(s)>0\}\cap [a,b]=\emptyset$, and $f_1$ is bounded away from zero.
Then $\xi(f_1)$ and  $\xi(f_2)$ are independent, and so are $\xi(f_1)$ and   $\xi( \phi_{n+1,k})$ .
In addition, by Lemma \ref{lemaa1} the law of $\xi(f_1)$  is equivalent to the Lebesgue measure.
Hence, from  \cite[Lemma 27.1-(iii)]{sato} it follows that
the law of the sum of the random variables $\xi(f_1)$ and $\xi(f_2+\phi_{n+1,k+1})$
is also equivalent to the Lebesgue measure.
Now, one easily prove  the assertion by an induction starting at $k=\min(J^\phi_{n+1})$.
\end{proof}

\begin{lemma}\label{lemmaa3}
For each $U\in \CB(G_n)$ and ${\by}\in \RR^{|J^\psi_n|}$, the conditioned measure
\DEQSZ\label{ldefined}
\CB(|J^\psi_n| )\ni U\mapsto l_n({\by},U)=\PP\lk({\bz}^n\in U\mid {\bg}^n ={\by}\rk)
\EEQSZ
is equivalent to the Lebesgue measure.
\end{lemma}

\begin{proof}
Given the scaling function $\phi$ there exists coefficients
$\{p_j: j=1,\ldots ,u\}$, where $u$ is the order of the Daubechies wavelet, such that
\DEQSZ\label{map1}
\phi(x)=\sum_{j=1} ^u p_j\phi(2x+j),\quad x\in\RR.
\EEQSZ
In addition,  we have the following representation
\DEQSZ\label{map2}
\phi(x)=\sum_{j=1} ^u (-1)^j p_j \psi(2x+j),\quad x\in\RR.
\EEQSZ
Because of the orthogonality of the wavelet basis we additionally have that
\DEQSZ\label{map3}
\sum_{j=1} ^k p_j \bar p_{j+2l} =\bcase
\frac 1 {\sqrt{2}}&\mbox{ for } l=j,\\
0& \mbox{ for } l\not =j.
\ecase
\EEQSZ
 Let us now consider the mapping
$$
\mathcal{I}:V_{n+1}\ni f\mapsto (f_{n+1,1},\ldots,f_{n+1,2^{n+1}})\in \RR^{|J^\phi_{n+1}|}, 
$$
where $f_{n+1,k}=\phi_{n+1,k}(f)$. It is not difficult to show that $\mathcal{I}$ is an isomorphism from $V_{n+1}$ onto $\RR^{|J^\phi_{n+1}|}$.
We note that since $V_{n+1}=V_n\otimes W_n$, it follows from \eqref{map1} that
there exists a linear mapping $T:V_{n+1}\to V_n$ which induces a mapping
$$
\mathcal{T}: \RR^{|J^\phi_{n+1}|}\to \RR^{|J^\phi_{n}|}.
$$
We can also define a mapping $\CS:V_{n+1}\to W_n$ by  $\CS {\bx} := \pi_{W_n} (I-\CT){\bf x}$.
As above we can also find a  linear mapping $S:V_{n+1}\to W_n$ inducing a mapping
$$
\mathcal{S}: \RR^{|J^\phi_{n+1}|}\to \RR^{|J^\psi_{n}|}.
$$
Since $V_{n+1}=V_n\otimes W_n$, we have $\CI ^{-1}\mbox{ker}(\mathcal{S})=V_n$ and $\CI ^{-1}\mbox{ker}(\mathcal{T})=W_n$.
Hence, from the Bayes formula we infer that
$$\PP\lk( \zeta={\bf x}\mid \gamma={\bf y}\rk)=\frac { \PP\lk( \pi_{W_n}\CI ^{-1}\mathcal{S}^{-1}{\bx} + \pi_{V_n}\CI ^{-1}\mathcal{S}^{-1}{\by}\rk)}{\PP\lk(  \CI ^{-1}\mathcal{S}^{-1}{\bf y}\rk)},
$$
  for any ${\bx}\in \RR^{|J^\psi_{n}|}$ and $ {\by}\in\RR^{|J^\phi_{n}|}$.
By Lemma \ref{exdensity} $\PP\lk(  \CI ^{-1}\mathcal{S}^{-1}{\by}\rk)>0$ and $\PP\lk( \pi_{W_n}\CI ^{-1}\mathcal{S}^{-1}{\bx} + \pi_{V_n}\CI ^{-1}\mathcal{S}^{-1}{\by}\rk)>0$. In particular,  
 there exists a density
$$
h_n({\bx},{\by})=\PP\lk( \zeta={\bx}\mid \gamma={\by}\rk),
$$
such  that
$$
l_n({\by},U) = \int_U h_n({\bx},{\by})\, d{\bx},
$$
and $h_n({\bx},{\by})>0$ for all  ${\bx}\in \RR^{|J^\psi_{n}|}$ and $ {\by}\in\RR^{|J^\phi_{n}|}$.
\end{proof}

In order to verify Assumption \ref{ass2} for a point $Y$
 we will show  in the following Lemma that for all $n\in\NN$, $\pi_{G_n} 0$ belongs to the support of the measure $\mu$.
 If this holds, we can set $Y=0$.

\begin{lemma} \label{lemma1}
Let $\alpha\in (0,2)$, $1\le p<\alpha$ and $s<\frac 1p-1$. Let $\nu$ be a $\sigma$--finite symmetric  measure on $\RR\setminus \{0\}$
such that there exists a number $\alpha\in (0,2]$ such that
$$ \nu(\RR\setminus[-\ep,\ep])\sim \ep^ {-\alpha}l(\ep) \quad \mbox{as} \quad \ep\to 0,
$$
for some slow varying function $l$.
Let  $\eta$ be the to $\nu$ corresponding Poisson random measure over the probability space $(\Omega,\CF,\PP)$.
Let $\mu$ be the from $\xi$ defined in \eqref{xidef} on $B_{p,p}^ s(\RR)$ induced probability measure.
Then for any $\ep>0$,
$$
\mu\lk( \{x\in B^ s_{p,p}(0,1): |x|_{B^ s_{p,p} }^p \le \ep\}\rk) >0.
$$
\end{lemma}
\begin{proof}
Let $L$ be given by
\DEQSZ\label{deflevy}
L(t):= \int_0 ^t \int_\RR z\, \eta(dz,ds) ,\quad t\in[0,1].
\EEQSZ

From \cite{aurzada}, Example 2.2 
we know that
$$
-\log\PP \lk( \sup_{0\le t\le 1}|L(t)|\le \ep\rk) \sim K \ep^ {\-\alpha},
$$
hence, for any $\tilde \ep>0$ there exists a $\delta>0$ such that
$$
\PP\lk( \sup_{0\le t\le 1}|L(t)|\le \tilde \ep\rk) \ge \delta .
$$

Let  $\tilde \ep>0$ be a constant to be chosen later and let us set
$$\Omega_{\tilde \ep}:= \lk\{  \sup_{0\le t\le 1}|L(t)|\le \tilde \ep\rk\}.
$$
Then,
\DEQS
\PP\lk( |\xi |_{B^s_{p,p}}^p \le \ep\rk)& =&
\PP\lk( |\xi |_{B^s_{p,p}}^p \le \ep\mid \Omega_{\tilde \ep}\rk)\PP\lk( \Omega_{\tilde \ep}\rk)
+
\PP\lk( |\xi |_{B^s_{p,p}}^p \le \ep\mid \Omega\setminus \Omega_{\tilde \ep}\rk)\PP\lk(  \omega\setminus\Omega_{\tilde \ep}\rk)
\\
&\ge &
\PP\lk( |\xi |_{B^s_{p,p}}^p \le \ep\mid \Omega_{\tilde \ep}\rk)\PP\lk( \Omega_{\tilde \ep}\rk)
\ge \delta \PP\lk( |\xi |_{B^s_{p,p}}^p \le \ep\mid \Omega_{\tilde \ep}\rk).
\EEQS
Note that on $\Omega_{\tilde \ep}$ the jump size of the process is less than $2\tilde \ep$. Hence
\DEQSZ\label{estimate1}\nonumber
\EE \lk[ |\zeta_{j,k}|^p \mid \Omega_{\tilde \ep}\rk] &\le &\EE \int_0^1\int_\RR \psi_{j,k}(s) 1_{|z|\le 2\tilde \ep}\eta(dz,ds)
\\ &\le& \frac {(2\tilde \ep)^{p-\alpha}} {p-\alpha}\, \int_0^1| \psi_{j,k}(s)|^p \, ds
\le  C_p \frac {(2\tilde \ep)^{p-\alpha}} {p-\alpha}\,2^{(\tfrac p2-1)j},
\EEQSZ
and
\DEQS
\EE \lk[ |\xi |^p_{B^s_{p,p}} \mid \Omega_{\tilde \ep}\rk] &\le &
\EE\lk[\sum_{j=0}^\infty 2   ^{j(s-\tfrac 1p)p}\sum_{k\in J^ \psi _j} |\zeta_{j,k}|^p  2^{ \tfrac {jp} 2}\mid \Omega_{\tilde \ep}\rk]
\\
&\le & C_p \frac {(2\tilde \ep)^{p-\alpha}} {p-\alpha}\,
\sum_{j=0}^\infty   2   ^{j(s-\tfrac 1p)p} \sum_{k\in J^ \psi _j} 2^{(\tfrac p2-1)j} 2^{ \tfrac {jp} 2}
\le \tilde C_p \frac {(2\tilde \ep)^{p-\alpha}} {p-\alpha}\,.
\EEQS
From these calculations we infer that
\DEQS
\PP\lk( |\xi |_{B^s_{p,p}}^p \le \ep\mid \Omega_{\tilde \ep}\rk)&=&1-\PP\lk( |\xi |_{B^s_{p,p}}^p >\ep\mid \Omega_{\tilde \ep}\rk)
\\ \ge 1-\frac {\EE  |\xi |^p _{B^s_{p,p}}}{\ep}
\ge 1-  \tilde{C_p} \frac {(2\tilde \ep)^{p-\alpha}} {\ep(p-\alpha)}.
\EEQS
Now, choosing $\tilde{\ep}$ such that
$$
\tilde C_p \frac {(2\tilde \ep)^{p-\alpha}} {\ep(p-\alpha)}=\frac 12,
$$
we infer that $$ \PP\lk( |\xi |_{B^s_{p,p}}^p \le \ep\mid \Omega_{\tilde \ep}\rk)\ge \frac12 ,$$
from
which the assertion follows.

\end{proof}

\noindent
For any $\DB\in \CB(B_{p,p}^s(\INT)$ we define the conditional probability $\mu(\,\cdot\,\mid \DB)$ by
$$
\CB(B_{p,p}^s (\INT))\ni U \mapsto \mu(U\mid \DB) := \bcase
                  \frac {\mu(U\cap \DB)}{\mu(\DB)}&\mbox{ if }  \mu(\DB)>0,
\\ 1&\mbox{ if }  \mu(\DB)=0.
\ecase
$$

\begin{lemma} \label{lemma2}
Let $\alpha\in (0,2)$, $1\le p<\alpha$ and $s<\frac 1p-1$. Let $\nu$ be a $\sigma$--finite symmetric  measure on $\RR\setminus \{0\}$
such that there exists a number $\alpha\in (0,2]$ such that
$$ \nu(\RR\setminus[-\ep,\ep])\sim \ep^ {-\alpha}l(\ep) \quad \mbox{as} \quad \ep\to 0,
$$
for some slow varying function $l$.

Let $\eta$ also be the  Poisson random measure, over the probability space $(\Omega,\CF,\PP)$, associated to the L\'evy measure $\nu$.
Let $\mu$ be the probability measure  on $B_{p,p}^ s(\INT)$  induced by the mapping $\xi$ defined in \eqref{xidef}.
Then, for any $R>0$, $x\in B_E(R,0)$  and 
$\ep>0$ there exist $n\in\NN$ and some $\delta>0$ such that 
$$
\mu \lk( \{ x\in B_{p,p}^ s (\INT): |\pi_{G^ n} x|_{B^ s_{p,p} }^p \le \ep\} \rk) >\delta.
$$
\end{lemma}
\begin{proof}
From Lemma \ref{lemma1} we infer that there exists a constant $\delta>0$ such that for
 $$\PP\lk( \DB_\Omega \rk) \ge \delta,$$
 where $\DB_\Omega := \lk\{ \sup_{0\le t\le 1} |L(t)|\le 1\rk\} $ and  $L=\{L(t):t\in[0,1]\}$ is
 defined in \eqref{deflevy},
Observe that the set $\DB:= \xi( {\DB}_\Omega)$  satisfies
 $\mu(\DB)\ge \delta$. Thus,
\DEQS
\lqq{ \mu\lk( \{ x\in B_{p,p}^ s (\INT): |\pi_{G^ n}x |_{B^s_{p,p}}^p \le \ep\}\rk)} &&
\\
& \ge &
\mu\lk( \{ x\in B_{p,p}^ s : |\pi_{G^ n}\xi |_{B^s_{p,p}}^p \le \ep\} \mid \DB\rk)\mu \lk( \DB 
\rk)
\ge \delta\cdot \mu\lk(  \{x\in B_{p,p}^ s :  |\pi_{G^ n}\xi |_{B^s_{p,p}}^p \le \ep\} \mid \DB\rk).
\EEQS
Now, from \eqref{estimate1} we infer that
\DEQS
\EE \lk[ |\pi_{G^ n}\xi |^p_{B^s_{p,p}} 1_{\Omega_{\tilde \ep}}\rk] &\le &
\EE\lk[1_{\Omega_{\tilde \ep}} \sum_{j=n+1}^\infty 2   ^{j(s-\tfrac 1p)p}\sum_{k\in J^ \psi _j} |\zeta_{j,k}|^p  2^{ \tfrac {jp} 2}\rk]
\\
&\le & C_p \frac {(2\tilde \ep)^{p-\alpha}} {p-\alpha}\,
\sum_{j=n+1}^\infty   2   ^{j(sp+p-1)}
\le \tilde C_p  2   ^{n(sp+p-1)}\sum_{j=0}^\infty   2   ^{j(sp+p-1)}
\le \hat  C_p  2   ^{n(sp+p-1)}.
\EEQS
Therefore,
\DEQS
\lqq{ \PP\lk( |\pi_{G^ n}\xi |_{B^s_{p,p}}^p \le \ep\mid \Omega_{\tilde \ep}\rk)=1-\PP\lk( |\pi_{G^ n}\xi |_{B^s_{p,p}}^p >\ep\mid \Omega_{\tilde \ep}\rk)}
&&
\\ &\ge& 1-\frac {\EE  |\pi_{G^ n}\xi |^p _{B^s_{p,p}}}{\ep}
\ge 1- \hat  C_p  2   ^{n(sp+p-1)}/\ep.
\EEQS
For  any $\kappa<1$ there exists a number $n\in\NN$ sufficiently large, such that
$$
 \hat  C_p  2   ^{n(sp+p-1)}/\ep\le 1-\kappa.
$$
which gives the assertion.

\end{proof}

\begin{lemma}\label{lemmaa6}
Let $\alpha\in (0,2)$, $1\le p<\alpha$ and $s<\frac 1p-1$. Let $\nu$ be a $\sigma$--finite symmetric  measure on $\RR\setminus \{0\}$
such that there exists a number $\alpha\in (0,2]$ such that
$$ \nu_j(\RR\setminus[-\ep,\ep])\sim \ep^ {-\alpha}l(\ep) \quad \mbox{as} \quad \ep\to 0,
$$
for some slow varying function $l$.

Then, for all $N\in\NN$,
 $x_0\in G_N$,   and all $\ep>0$, there exists a $\delta>0$ such that
$$
\mu\lk( \lk\{ x\in B^ s_{p,p}(\INT)\mid |x-x_0|_{B^ s_{p,p}}\le \ep\rk\}\rk)\ge \delta.
$$
\end{lemma}

\begin{proof}

Let $\ep>0$ be a fixed constant and  $s<s_0<\frac 1p-1$.
From Lemma \ref{lemma2} we deduce that there exist $n_0\in\NN$ and $\delta_2>0$ such that
$$
\mu\lk( \lk\{ x\in B^ {s}_{p,p}(\INT)\mid |\pi_{G^ {n_0}}x|_{ B^ {s_0}_{p,p}}\le \frac \ep 4\rk\}\rk)\ge \delta_2.
$$
Then
\DEQS
\lqq{\mu\lk( \lk\{ x\in B^ {s}_{p,p}(\INT)\mid |x_0-x|_{ B^ {s}_{p,p}}\le  \ep \rk\}\rk)
}
\\
&\ge&
  \mu\lk( \lk\{ x\in B^ {s}_{p,p}(\INT) 
 \mid |x_0-\pi_{G_{n_0}} x|_{ B^ {s}_{p,p}}\le \frac \ep 4 \rk\}\cap \lk\{ x\in B^ {s}_{p,p}(\INT) 
  \mid |\pi_{G^ {n_0}} x|_{ B^ {s_0}_{p,p}}\le \frac \ep 4 \rk\}\rk).
\EEQS
We now set  $A_{n_0}=\lk\{ x\in  B^ {s}_{p,p}(\INT)\mid |x_0-\pi_{G_{n_0}} x|_{ B^ {s}_{p,p}}\le \frac \ep 4 \rk\}$ and observe that for  $\gamma=\delta_2/2>0$ there exists a closed set $C_\gamma\subset G^ {n_0}$ such that  $\mu^ {n_0}(G^ {n_0}\setminus C_\gamma)\le \gamma$ and
the function
$$
C_\gamma \ni \by \mapsto l_{n_0}(\by,A_{n_0})\in[0,1]
$$
is continuous. Furthermore,  since for all $\by\in G^ {n_0} $ $\mu$ a.s.\ $l_n(\by,\cdot)$ is equivalent to the Lebesgue measure and $\leb_{G_{n_0}}({A_{n_0}})>0$, we have
 $l_{n_0}(\by,A_{n_0})>0$.
Since the embedding $ B^ {s_0}_{p,p}(\INT)\hookrightarrow  B^ {s}_{p,p}(\INT)$ is compact,
$$
C_{n_0}=\lk\{ x\in B^ {s}_{p,p}(\INT) 
  \mid |\pi_{G^ {n_0}} x|_{ B^ {s_0}_{p,p}}\le \frac \ep 4 \rk\}\cap C_\gamma
$$
 is a compact subset of $G^ {n_0}$ and there exists a $\delta_3>0$ such that
for all
$ \by\in  C_{n_0}\cap C_\gamma $, 
 $l_{n_0}(\by,A_{n_0})\ge \delta_3$.
From the above consideration we now infer that
\DEQS
\lqq{ \mu\lk(\lk\{ x\in  B^ {s}_{p,p}(\INT) \mid |x-x_0|\le \ep\rk\}\rk)
\ge \int_{\{|\pi_{G^ {n_0}} x|_{ B^ {s_0}_{p,p}}\le \tfrac \ep 4\}\cap C_\gamma } l_{n_0}(\by,A_{n_0})\, \mu^ {n_0}(d\by)
}&&
\\
&\ge& \delta_3 \mu^  {n_0}\lk(\lk\{ |\pi_{G^ {n_0}} x|_{ B^ {s_0}_{p,p}}\le \tfrac \ep 4\rk\}\cap C_\gamma\rk)
\\
&\ge &
\delta_3 \lk( 1- \mu^  {n_0}\lk(\lk( G^ {n_0}\setminus \lk\{ |\pi_{G^ {n_0}} x|_{ B^ {s_0}_{p,p}}\le \tfrac \ep 4\rk\}\rk) \cup \lk( G^ {n_0}\setminus C_\gamma\rk)\rk)\rk)
\\
&\ge& \delta_3\lk(1- \lk( 1-\delta_2+\gamma\rk)\rk)= \delta_3\frac {\delta_2}2 .
\EEQS
\end{proof}

The above discussion is summarized in the following lemma.

\begin{lemma}\label{lemma11}
Let $\eta$ be a time homogeneous Poisson random measure on $\RR$ over a probability space $(\Omega,\CF,\PP)$. We assume that the L\'evy measure $\nu$ associated to $\eta$ is symmetric, $\sigma$-additive, absolutely continuous with respect to the  Lebesgue measure on $\RR\setminus \{0\}$.
In addition, we assume, that there exists some $p\in(1,2)$ with
 $$\int_{|z|\le 1} |z|^ p\nu(dz)<\infty $$
and there exists a number $\alpha\in (0,2]$ such that
$$ \nu(\RR\setminus[-\ep,\ep])\sim \ep^ {-\alpha}l(\ep) \quad \mbox{as} \quad \ep\to 0,
$$
for some slow varying function $l$.

Let $\{ \phi_{j,k}:j\in\NN:k=1,\ldots ,2^ j\}$ be the wavelet  basis in $B_{p,p}^ {s}(\INT)$ described in Section \ref{haar}.
Then, the measure $\mu$ induced by the map $\xi$ defined by \eqref{xidef} on  $B_{p,p}^ {s}(\INT)$
is decomposable with decomposition $\{F_n,G^n,l_n\}_{n=0}^\infty$ satisfying Assumption \ref{ass1} and Assumption \ref{ass2}. Here
 the spaces $F_n$ are defined by $F_0=V_0$, $F_n=W_{n}$, $n\ge 2$,   $V_0$ and $W_n$
are defined in  \eqref{exactspaceslevy}, and $l_n$ is defined in \eqref{ldefined}. 
\end{lemma}
\begin{proof}
The decomposability follows from the fact the wavelet basis  described in section \ref{haar}   is a Schauder basis of $B_{p,p}^ {s}(\INT)$.
Assumption \ref{ass2}-(1) follows from  choosing  $Y=(0,0,\ldots)$ and from the fact that
 $\PP\lk( \pi_{G_{j+1}} x\in \cdot \mid \pi_{G_j}x=y\rk)$ is equivalent 
 to the Lebesgue measure and
that for any $y\in\RR$ we have (see Lemma  \ref{lemmaa3})
 $$\PP\lk( \pi_{G_{j+1}} x\in \cdot \mid \pi_{G_j}x=y\rk)>0.$$
Using an induction argument one can easily show that for any open set in $G_n$ $\mu_{G_n}(\CO)>0$ from which it follows that $\mu_{G_n}$ is absolutely continuous 
 with respect to $\leb_{G_n}$. Finally,  Assumption \ref{ass2}-(2)  follows from Lemma \ref{lemmaa6}.
\end{proof}

\section{The Fractional Brownian Noise and its Wavelet Expansion}
\label{haartrue}

Let $B^ H=\{ B^ H(t):t\ge 0\}$ a the fractional Brownian motion with Hurst parameter $H\in(\frac 12,1)$. Let us fix $s\in (-\frac12, H-1)$ and consider the mapping
%
$$
\xi^H:\CS([0,1]) 
\ni \phi \mapsto \xi^H(\phi)=\int_0^ 1 \phi(t)\,dB^ H(t).
$$
For all $n\in\NN$, $\phi_1,\ldots,\phi_n\in \CS(\RR )$, 
 $C\in\CB(\RR^n)$ we set
\DEQSZ\label{measurefractal}
\mu\lk( \lk\{ x\in \CS'( 
[0,1]): (x(\phi_1),\ldots,x(\phi_n))\in C\rk\}\rk) := \PP\lk(  (\xi(\phi_1),\ldots,\xi(\phi_n))\in C \rk).
\EEQSZ


We will firstly show that this measure $\mu$ is a Radon measure on $B_{2,2}^{s}(\INT)$. For this aim, let us consider the  Haar wavelet $\psi$  defined by
$$
\psi(t) := \bcase 1; & \mbox{ for } t\in[0,\frac 12),\\
-1; & \mbox{ for } t\in[\frac 12,1],
\\
0 & \mbox{ elsewhere},
\ecase
$$
and  the scaling function $\phi$ defined by
$$
\phi(t) := \bcase 1; & \mbox{ for } t\in[0,1],\\
0; & \mbox{ elsewhere}.
 \ecase
$$
Also, we set
\begin{equation}\label{wavelet-pa}
\psi_{j,k}:=2^{-\frac j2}\psi(2 ^jt+k)\mbox{ and }  \phi_{j,k}:=2^{-\frac j2}\phi(2 ^jt+k),  j=1,\ldots,n,\,\, k=1,\ldots, 2 ^j,
\end{equation}
to which we associate the  multiresolution analysis
\DEQSZ\label{exactspacesbr}
V_n:= \spn\{ \phi_{j,k}: j=1,\ldots,n,\,\, k=1,\ldots, 2 ^j\},
\quad
W_n:= \spn\{ \psi_{n,k}:   k=1,\ldots, 2 ^n\}.
\EEQSZ


The Haar wavelet is an unconditional basis in $L^p(\INT )$ with $1<p<\infty$,
a basis in $B_{p,q}^s(\INT )$ for $1<p<\infty$ and $\frac 1p-1<s<\frac 1p$, and
a basis for $B_{p,p}^s(\INT )$,  $\frac 12 <p\le 1$ and $\frac 1p-1<s<1$ (see Triebel \cite[Theorem 1.58]{triebel_fkt}).

\medskip

Now, let
$F_0:=V_0$, $F_n=W_n$, $n\in\NN$, and $G_n:= F_0\oplus F_1\oplus\cdots \oplus F_n$, and
$\CF_n := \sigma( F_0\oplus F_1\oplus\ldots \oplus F_n)$.
Let us denote the projection of $\xi$ 
 onto $G_n$ by $P_n$
and  onto $W_n$ by $Q_n$.
For the time being let us assume that the Radon-Nikodym derivative of the fractional Brownian motion belongs to $ B^ s_{2,2}(\INT)$.
Since the Haar wavelets are a  basis of $E:= B^ s_{2,2}(\INT)$, then for each element $x\in E$ there exists a unique sequence  $\{ \lambda_{j,k}: j\in\NN, k=1,\ldots, 2^ j\}$ such that
$$
x=\sum_{j\in\NN} \sum_{k=1}^ {2^n} \lambda_{j,k} \psi_{j,k} + \lambda _0\phi .
$$
Observe also that

$$
\xi(t) := \sum_{j=1}^\infty \sum_{k=1}^{2^j} \zeta_{j,k} \psi_{j,k}(t)+a_0\phi(t) ,\quad t\in[0,1].
$$
where $\{ \zeta_{j,i}:j\in\NN, i=1,\ldots,2 ^j\}$ is a family of
random variables defined by
$$
\zeta _{j,k}\stackrel{d}{=} \int_0^1  \psi_{j,k}(s) \, dB^ H(s),
$$
and
$$
a_0 \stackrel{d}{=} \int_0^1  \phi_{0,0}(s) \,dB^H(s).
$$
In fact, given the coefficients $\{ \zeta_{j,k}: j=1,\ldots,n, k=1,\ldots , 2^ j\}\cup\{ a_0\}$, one know the
coefficient of $\phi_{n,k}$, $k=1,\ldots,2^n$.
Since $G_n$ consists of all functions $f:[0,1)\to\RR$ that are constant on the intervals $[2^ {-n}k,2^ {-n}(k+1))$, $k=1,\ldots,2^ n-1$,
there exists random coefficients $\gamma_{n,k}$, $k=1,\ldots,2^ n-1$ such that
$$
P_n \xi =\sum_{k=0}^ {2^ n-1}\gamma_{n,k} \phi_{n,k}.
$$
It is now easy to see show that
$$
\gamma_{n,k}:= \int_0^1 \phi_{n,k}(t)dB^H(t).
$$

Since for two functions $\phi,\psi:[0,1]\to\RR$, the random variables $\xi^H(\phi)$ and $\xi^H(\psi)$
are Gaussian distributed with covariance 
$$
\EE \lk[ \xi^H(\phi)\, \xi^H(\psi)\rk]=\int_0^ 1\int_0^ 1\phi(s)\phi(t)\, |t-s|^ {2H-2}\, dt\,ds,
$$
straightforward calculations gives for $l\not = k$
\DEQS
\lqq{ \EE \lk[ \xi^ H(\psi_{j,k})\,\xi^ H(\psi_{j,l})\rk]=2^ {j}\int_{2^  {-j}k}^{2^  {-j}(k+1)}\int_{2^  {-j}l}^{2^  {-j}(l+1)}\psi_{j,k}(s)\psi_{j,l}(t)|t-s|^ {2H-2}\, dt\,ds}&&
\\
&=&
2^ {j}\frac 1 {2H-1}\int_{2^  {-j}k}^{2^  {-j}(k+1)}\lk[ (t-2^  {-j}l)^ {2H-1} - (t-2^  {-j}(l+1))\rk]\, dt
\\
&=&
2^ {j}\frac 1 {2H-1}\frac 1 {2H}\lk\{ \lk[ (2^  {-j}k-2^  {-j}l)^ {2H-1} \rk.\rk.
\\
&&{}\lk.\lk.{}- (2^  {-j}k-2^  {-j}(l+1))\rk]-\lk[ (2^  {-j}(k+1)-2^  {-j}l)^ {2H-1} - (2^  {-j}(k+1)-2^  {-j}(l+1))\rk]\rk\}
\\
&=&
2^ {j}\frac 1 {2H-1}\frac 1 {2H}\lk\{ \lk[ 2(2^  {-j}k-2^  {-j}l)^ {2H-1} \rk.\rk.
\\
&&{}\lk.\lk.{}- (2^  {-j}k-2^  {-j}(l+1)) - (2^  {-j}(k+1)-2^  {-j}l)^ {2H-1}\rk]\rk\}\sim 2^  {-j}|k-j|^ {2H-1} 2^  {-j}.
\EEQS
Hence
$$
\EE \zeta_{j,k}\zeta_{j,l} \sim 
 2^  {-j}|k-j|^ {2H-1} .
 $$
One can also easily prove that for $l = k$
\DEQS
\lqq{ \EE \lk[ \xi^ H(\psi_{j,k})\,\xi ^ H(\psi_{j,k})\rk]}
&&
\\
&=&2^ {j}\int_{2^  {-j}k}^{2^  {-j}(k+1)}\int_{2^  {-j}k}^{2^  {-j}(k+1)}\psi_{j,k}(s)\psi_{j,l}(t)|t-s|^ {2H-2}\, dt\,ds
\\
&=&
2^ {j}\frac 1 {2H-1}\frac 1 {2H}\lk\{ \lk[ (2^  {-j}k-2^  {-j}l)^ {2H-1} \rk.\rk.
\sim 2^ {1-2Hj}.
\EEQS
Using these estimates we can prove the following proposition.

\begin{proposition}\label{fbmex}
For $H\in(\frac 12,1)$ and $-\frac 12 <s<H-1$
 we have $\xi^ H\in L^ 2(\Omega;B_{2,2}^ s(\INT))$.
\end{proposition}
\begin{proof}
The proof is the result of the following  straightforward calculation
\DEQS
\EE |\xi|_{B^ s _{2,2}}^ 2 =\EE \sum_{j=0}^ \infty 2^ {2sj} \sum_{k=0}^ {2^ j}\EE \, \zeta_{j,k}\lesssim \sum_{j=0}^ \infty 2^ {2sj} {2^ j}2^ {1-2Hj}
          \lesssim
 \sum_{j=0}^ \infty 2^ {j(2s+2-2H)} {2^ j}2^ {1-2Hj}
\EEQS
Now, the sum is finite if $s+1-H<0$.
\end{proof}

\begin{remark}\label{remarkBH}
If $H\in(\frac 12,1)$ one can find a number $s\in(-\frac 12,H-1)$ such that $\xi^H \in L^ 2(\Omega;B_{2,2}^ s(\INT ))$. Since  all coefficients of $\phi_{n,k}$ and $\psi_{n,k}$ 
are Gaussian distributed,  their law are equivalent with respect to the Lebesgue measure. Now, since the Haar basis is a Schauder basis in $B_{2,2}^ s(\INT )$, Assumption \ref{ass1}
is  satisfied.
By the same arguments as used in the proof of Lemma \ref{lemmaa6}, one can show that
Assumption \ref{ass2} is also satisfied.
\end{remark}

\medskip

\begin{lemma}\label{lemma11fbm}
Let $B^H$ be a fractional Bownian motion with Hurst parameter $H>\frac 12$ and $\mu$ the probability measure on $B_{2,2}^s(\INT)$ defined by \eqref{measurefractal}.
 Let $\{ \phi_{j,k}:j\in\NN:k=1,\ldots ,2^ j\}$ be the wavelet  basis in $B_{2,2}^s(\INT)$ described in \eqref{wavelet-pa}.
Then, the measure $\mu$ is decomposable with decomposition satisfying Assumption \ref{ass1} and Assumption \ref{ass2}.
\end{lemma}

\section{Zero One Laws for decomposable measures with density}

In this Section we generalize the Theorem 4 of  \cite{dineen} to
decomposable measures with decomposition as defined in Definition \ref{decomposkerneldef}. We will also identify the conditions under which a measure satisfies Assumption \ref{ass11} and Assumption \ref{ass1}.

Throughout this section  $E$ denotes and arbitrary a separable Banach space and
$\CB(E)$ the $\sigma$--algebra generated by its open sets.
Let $\mu$ be a measure on $(E,\CB(E))$ and $F$ and $G$ be two subsets of $E$ such that  $E=F \oplus G$.
 Then, there is a probability measure
$$
\mu_{(F,G)}:\CB(F)\ni A\mapsto \mu(A+G)\in[0,1].
$$
For $A\subset E$ and $y\in G$ let $A_{(F,G)}(y)=\{ x\in F: x+y\in A\}$.

As mentioned in the introduction the concept of decomposability can be extended to the notion of decomposability we introduced in Definition \ref{decomposkerneldef}.

\begin{example}\label{bevor}
Let $E$ be a separable Banach space and
 $\{e_n:n\in\NN\}$ be a Schauder basis  and $F_n:= \{ \lambda e_n:\lambda \in\RR\}$.
For each element $x\in E$ there exists a unique sequence $\{ a_n: {n\in\NN}\}$ in $\RR$ such that
$x=\sum_{n\in\NN} a_n e_n$. Let $G_n:= F_1\oplus\cdots\oplus F_n$, $G^ n=G_n^ \perp$,
$$\pi_{G_n}:E\ni x\mapsto a_1e_2+ \cdots +a_ne_n \to G_n
$$
be  a projection from $E$ onto $F_1\oplus\cdots\oplus F_n$ and
$$
\pi_{G^ n}: E\ni  x\mapsto \sum_{j\in\NN} a_{j+n}e_{j+n} \in G^n.
$$
Then, the probability measure of each $E$--valued random variable is decomposable in the sense of Definition \ref{decomposkerneldef}.
This can be shown by the following consideration.
From the Radon-Nikodym Theorem (see \cite[Theorem 6.3]{kallenberg})
for any $E$--valued random variable $X$
there exists
a probability kernel
$$
l_n:G^n \times \CB(F_1\oplus\cdots\oplus F_n) \to [0,1],
$$
such that
\begin{enumerate}
  \item
$\PP\lk( \pi_{G_n}X\in U\mid \pi_{G^ n}X=y\rk)=l_n(y,U)$ for all $U\in \CB(F_1\oplus\cdots\oplus F_n)$;
  \item for each $U\in \CB(F_1\oplus\cdots\oplus F_n)$ the mapping
$$
G^n\ni y \mapsto l_n(y,U)
$$
is $\CB(G^n)$--measurable
\end{enumerate}
\end{example}

To simplify the notation let us denote $\mu_{(G_n,G^ n)}$ by $\mu_n$ and $\mu_{(G^ n,G_ n)}$
by $\mu^n$.
Note that given a decomposition $(F_n,G_n,l_n)$  of $\mu$ it is essential that the kernel $l_n$ has a density with respect to the Lebesgue measure on $G_n$ which, as we will show in the next Lemma, follows from the absolute continuity of $\mu_n$ with respect to $\leb_{G_n}$ for any $n\in\NN$.

\del{\begin{lemma}
In the setting of Example \ref{bevor} the following holds:
For all $n\in\NN$, for any $A\in\CB(F_1\oplus\cdots \oplus F_n)$  and all $\ep>0$ there exists a closed subset $C_{n,A}^\ep$ of $G^n$ such that
$\mu^n(G^n\setminus C_{n,A}^\ep)\le \ep$ such that the function
$$
l_n:C^\ep_{n,A} \ni y \mapsto l_n(y,A)\in  [0,1],
$$
is continuous.
\end{lemma}
\begin{proof}
This follows by the fact that the measure $\mu$ is perfect, hence the measure $\mu^n$ on $G^n$ and
the fact that for all $A\in \CB(F_1\oplus\cdots \oplus F_n)$ the mapping
$$ l_n:G^n \ni y \mapsto l_n(y,A)\in  [0,1],
$$ is measurable, hence, by Theorem 4.1 \cite[p. 41]{parth} continuous.
\end{proof}
}
\medskip

\begin{lemma}\label{abscont}
Let $E$ be a separable Banach space and
 $\{e_n:n\in\NN\}$ be a Schauder basis  and $F_n:= \{ \lambda e_n:\lambda \in\RR\}$, $G_n:=F_1\oplus \cdots \oplus F_n$.
Let us assume that for all $n\in\NN$  $\mu_{G_n}$ is absolutely continuous with respect to  $\leb_{G_n}$.
Then for any $n\in\NN$,
$$\mu^ n( \{ y\in G^n: l_n(y,\cdot) \mbox{ is abs. continuous with respect to } \leb_{G_n}\})=1.
$$
In particular,
for any $U\in \CB(G_n)$ with $\mu_n(U)=0$, we have
$$
\mu^n( \{ y\in G^n : l_n(\by,U)=0\})=1.
$$
\end{lemma}
\begin{proof}
Fix $U\in \CB(G_n)$ with $\mu_{G_n}(U)=0$.
We will show that
$\mu^n( \{ \by\in G^n : l_n(y,U)>0\})=0$.
From  \cite[Theorem 4.1 ]{parth} we infer that for all $\ep>0$ there exists a subset $C_{n,U}^\ep\subset G^n$ such that $\mu^ n(G^n\setminus C_{n,U}^  \ep)\le \ep$ and
$$
l_n(\cdot,A):C^\ep_{n,U} \ni \by \mapsto l_n(\by,U)\in  [0,1],
$$
is continuous.
Now let us set
$$
G_\ep ^\ast =\{ \by\in G^n\cap C_{n,U}^\ep  : l_n(\by,U)\ge \ep\}.
$$
Since $l_n(\cdot,U)\big|_{C^\ep_{n,A}}$ is continuous and the sets $[\ep,1]$  and  $C_{n,U}^\ep$ are closed, the set $G_\ep^\ast$ is closed. Hence,
$$
0=\mu_{G_n}(U)= \mu(U+G^n)=\int_{G^n} l_n(\by,U)\,\mu^n(d\by).
$$
Since $G^ n\supset G_\ep ^\ast $, we additionally have
$$
0=\mu_n(U)=\mu(U+G_n)=\int_{G^n} l_n(\by,U)\,\mu^n(d\by)\ge
\int_{G_\ep ^\ast } l_n(\by,U)\,\mu^n(d\by).
$$
By the definition of the set $G_\ep ^\ast$ we have
$$
0\ge \ep  \, \mu^n({G_\ep ^\ast }).
$$
Since $\ep >0$, we have $ \mu^n({G_\ep ^\ast })=0$.
Now, from the closedness of $G_\ep ^\ast$  and the regularity of the measure$\mu_{G^n}$  we infer that
$$
\mu^n (\{ \by\in G^n: l_n(\by,U)>0\})=\lim_{\ep \to 0} \mu_{G^n}(G_\ep ^\ast)=0.
$$
\end{proof}

\begin{lemma}\label{abscontfff}
Let $\mu$ be a decomposable finite measure on $E$ with decomposition $\{F_n,G^n,l_n\}_{n=1}^\infty.$
Let us assume that $\mu_{G_n}$ is absolutely continuous with respect to $\leb_{G_n}$. {Then, for
 any $U\in \CF_n$ satisfying $\mu_{G_n}(U)=0$ we have
$$
\mu_{G^n}( \{ \by\in G^n : l_n(\by,U)=0\})=1.
$$}
\end{lemma}
\begin{proof}
Let $n\in\NN$ and $U\in \CF_n$ such that $\mu_{G_n}(U)=0$. We will show that
$\mu_{G^n}( \{ \by\in G^n : l_n(\by,U)>0\})=0$.

\medskip
Firstly, note  that
by the Radon-Nikodym Theorem the mapping
$$
l_n:C^\ep_{n,U} \ni y \mapsto l_n(y,U)\in  [0,1],
$$
is measurable.
Hence, from \cite[Theorem 4.1]{parth} we infer that for all $\ep>0$ 
there exists a closed subset $C_{n,U}^\ep$ of $G^n$ such that
$\mu^n(G^n\setminus C_{n,U}^\ep)\le \ep$ and the function
$$
l_n:C^\ep_{n,U} \ni \by \mapsto l_n(\by,U)\in  [0,1],
$$
is continuous.
\del{This follows by the fact that the measure $\mu$ is perfect, hence the measure $\mu^n$ on $G^n$ and
the fact that for all $U\in \CB(F_1\oplus\cdots \oplus F_n)$ the mapping
$$ l_n:G^n \ni y \mapsto l_n(y,U)\in  [0,1],
$$ is measurable, hence, by Theorem 4.1 \cite[p. 41]{parth}, there exists a closed subset $C_{n,U}^\ep$ of $G^n$ such that $l_n(\cdot,U)|_{C_{n,U}^ \ep}$ is
  continuous and $\mu_n(C_{n,U}^ \ep)\ge 1-\ep$.
}
\medskip
Secondly, let us set
$$
G_\ep ^\ast =\{ \by\in C_{n,U}^\ep  : l_n(\by,U)\ge \ep\}.
$$
From the continuity of $l_n(\cdot,U)\big|_{C^\ep_{n,U}}$  and the fact that the sets $[\ep,1]$ and
 $C_{n,U}^\ep$ are closed we conclude the set $G_\ep^\ast$ is also closed.
Next, thanks to the definition of $\mu_{G_n}$ we obtain that
$$
0=\mu_{G_n}(U)= \mu(U+G^n)=\int_{G^n} l_n(\by,U)\,\mu_{G^n}(d\by).
$$
Furthermore, because $G_\ep ^\ast\subset C_{n,U}^\ep $ we also have
$$
0=\int_{G^n} l_n(\by,U)\,\mu_{G^n}(d\by)\ge
\int_{G_\ep ^\ast } l_n(\by,U)\,\mu_{G^n}(d\by).
$$
Invoking now the definition of the set $G_\ep ^\ast$ we obtain
$$
0\ge \ep  \mu_{G^n}({G_\ep ^\ast }).
$$
Since $\ep >0$, we have $\mu_{G^n}({G_\ep ^\ast })=0$.
From the closedness of $G_\ep ^\ast$  and the regularity of the measure $\mu^n$ we infer that
$$
\mu_{G^n} (\{ \by\in G^n: l_n(\by,U)>0\})=\lim_{\ep \to 0} \mu_{G^n}(G_\ep ^\ast)=0.
$$
Therefore,
$$
\mu_{G^n} (\{ \by\in G^n: l_n(\by,U)=0\})=1.
$$

\end{proof}


\begin{cor}\label{densityex}
Let $E$ be a separable Banach space and
 $\{e_n:n\in\NN\}$ be a Schauder basis. Put $F_n:= \{ \lambda e_n:\lambda \in\RR\}$ and  $G_n:=F_1\oplus \cdots \oplus F_n$.
Let us assume that for all $n\in\NN$ $\mu_{G_n}$ is absolutely continuous with respect to the $\leb_{G_n}$.
Then for any $n\in\NN$ there exists a function $h_n: G^n\times F_1\oplus\cdots\oplus F_n\to\RR^+_0$ such that
$\mu_{G^n}$--a.s.\
$$
l_n(\by,U) = \int_U h_n(\by,x)\mu_{G_n}(dx).
$$
\end{cor}
\begin{proof}
From
$$
\mu_{G^n} (\{ \by\in G^n: l_n(\by,U)=0\})=1, \text{ for any } U\in \CB( G_n),
$$
follows the corollary's assertion.
Indeed the above identity implies the existence of a Radon-Nikodyn derivative. In particular, it holds that
\DEQS
\mu_{G^n} \Big(\big \{ \by\in G^n: \mbox{ there exists a mapping $h_n(\by,\cdot):G_n\to \RR$}\hspace{4cm}
\\ \hspace{5cm} \mbox{  such that } l_n(\by,U)=\int_U h_n(\by,x)\mu_{G_n}(dx)\big\}\Big)=1.
\EEQS
\end{proof}

\begin{definition}
We call a set $U\in \CB^\mu(E)$ a finite zero one $\mu$--set if and only if for all $n\in\NN$
$$
\mu_{G^n} \lk\{ \by\in G^n: \mu_{G_n}(U_n(\by))=0\mbox{ or } 1 \rk\}=1,
$$
where $U_n(\by)=U_{(F_1\oplus\cdots\oplus F_n,G^n)}(\by)$.

\end{definition}

\del{
\begin{example} Let us assume that the measure on $(F_n,\CB(F_n))$ induced by $\pi_n^{-1}\circ \mu$, where   $\pi_n:E \mapsto F_n$
is absolutely continuous with the respect to the Lebesgue measure. Then
$$ A=\cup_{n\in\NN}\{ \lambda e_n:\lambda\in\QQ\} $$
or
$$ A=\cup_{n\in\NN}\{ \lambda e_n:\lambda\in\RR\setminus \QQ\} $$
are  finite zero one $\mu$--sets.
\end{example}
}

 Let us now present the generalization of Theorem 4 in \cite{dineen}.

\begin{theorem}\label{theorem4dineen}
Let  $\{ F_n,G^n, l_n\}_{n=1}^\infty$ be a decomposition for $\mu$  such that for any $n\in\NN$
$\mu_{G_n}$ is absolutely continuous with respect to $\leb_{F_1\oplus\cdots\oplus F_n}$. Let $F_\infty = \cup_{n\in\NN} \{ F_1+F_2+\cdots + F_n\}$.
If $U$ is a finite zero one $\mu$ measurable subset of $E$, then there exists $B\in\CB(E)$ such that $B+F_\infty=B$ and $\mu(B)=\mu(U)$.

\del{Moreover,
\begin{enumerate}
  \item if for all $n\in\NN$ and $y\in G^n$
  $\mu_{G_n}(U_n(y))=1$ implies $U_n(y)=F_1+F_2+\cdots +F_n$,
  then we can choose $B\subset U$.

  \item if for all $n\in\NN$ and $y\in G^n$
  $\mu_{G_n}(U_n(y))=0$ implies $U_n(y)=\emptyset$,
  then we can choose $B\supset U$.

  \item $\{ F_n,G^n\}_{n=1}^\infty$ is also for $\nu$ a decomposition and $\nu$ is finitely equivalent to $\mu$ and $U$ is a finite zero one set $\nu$ set, then $\nu(B)=\nu(U)$.
\end{enumerate}
}

\end{theorem}

\begin{proof}
The proof is very similar to the proof of \cite[Theorem 4]{dineen}.
Let us assume $U\in\CB(E)$. For fix $n\in\NN$ we set
 $U^n = \lk\{ y\in G^n: \mu_{G_n}(U_n(y))=1\rk\}$,
$$G_n=F_1\oplus F_2\oplus\cdots \oplus F_n = \mbox{linear span of} \cup_{k=1}^n F_k,
$$
 $B_n=G_n+U^n$ and $B=\liminf_{n\to\infty} B_n = \cup_{n=1}^ \infty \lk\{ \cap_{m\ge n} B_m\rk\}$.
For the time being let us assume that
\DEQSZ\label{claim}
\mu(U)=\mu(B_n).
\EEQSZ
 Then,
\begin{itemize}
  \item $\mu(U \bigtriangleup B_n)=0$ for all $n\in\NN$,
  \item and $\mu(U)=\mu(B_n)\ge \mu(\cap _{m\ge n} B_m) \ge \mu(U)$,
  \item $\mu(B)=\lim_{n\to\infty} \mu(\cap_ {m\ge n}B_m)$.
\end{itemize}
Since $\mu$ is regular we additionally have that
$$\mu(B)=\lim_{n\to\infty} \mu(\cap_ {m\ge n}B_m) \ge \lim_{n\to\infty} \mu(B_n)=\mu(U),
$$
from which the assertion of Theorem \ref{theorem4dineen} follows.

\medskip

Now it remains to prove \eqref{claim}. To this end,  observe first that because of Lemma \ref{abscont} the kernel $l_n$  is $\mu^ n$--a.s.\  absolutely continuous on $G_n$.
Hence, by the Radon-Nikodym Theorem for $\mu^ n$--a.s.\  there exists
a probability kernel
$$
h_n:G^n\times G_n 
 \to \RR^ +_0, 
$$
such that
$$
\mu(U) = \int_{G^n} \int_{U_n(y)} 
 h_n(\by,x)\,\mu_{G_n}(dx) \mu_{G^ n} (d\by) . 
$$
Then, by using $B_n=G_n\oplus U_n$ we obtain that
\DEQS
\lqq{ \mu(U)=\int_{G^n}\int_{G_n} 1_U(x+\by) h_n(y,x)\, \mu_{G_n}(dx)\,\mu_{G^ n}(d\by)
}
&&
\\
&=&\int_{G^n} \int_{G_n} 1_{U_n(\by)\oplus U^ n } (x+y) h_n(\by,x)\, \mu_{G_n}(dx)\,\mu_{G^ n}(d\by)
\\
&=&\int_{G^n} \int_{G_n} 1_{(U_n(\by)\oplus U^ n)\cap B_n } (x+\by) h_n(\by,x)\, \mu_{G_n}(dx)\,\mu_{G^ n}(d\by)
\\
&=&\int_{G^n} \int_{G_n} 1_{( U_n(\by) \cap G_n) \oplus U^ n } (x+\by) h_n(\by,x)\, \mu_{G_n}(dx)\,\mu_{G^ n}(d\by)
\\
&=&\int_{U^ n} l_n(\by, U_n(\by) \cap G_n ) \, \mu_{G_n}(dx)\,\mu_{G^ n}(d\by)
\\&=& \int_{U^ n} l_n(\by,  G_n ) \, d\mu_{G_n}(x)= \mu(B_n).
\EEQS

\end{proof}


\end{document}